\documentclass[a4paper,10pt]{amsart}

\usepackage{pst-all}
\usepackage[latin1]{inputenc}
\usepackage{amssymb,amsmath,amsthm,mathrsfs,bbm}
\usepackage{graphicx}

\psset{unit=1cm}

\newcommand{\R}{\ensuremath{\mathbb{R}}}
\newcommand{\Z}{\ensuremath{\mathbb{Z}}}
\newcommand{\C}{\ensuremath{\mathbb{C}}}
\newcommand{\N}{\ensuremath{\mathbb{N}}}
\newcommand{\Proj}{\ensuremath{\mathbb{P}}}
\newcommand{\A}{\ensuremath{\mathbb{A}}} 
\newcommand{\xin}{\ensuremath{x_1,\ldots,x_n}}
\newcommand{\vek}[1]{\mathbf{#1}}   
\newcommand{\Id}{\ensuremath{\mathscr{I}}}  
\newcommand{\Jd}{\ensuremath{\mathscr{J}}}  
\newcommand{\mc}[1]{\ensuremath{\mathcal{#1}}}   
\newcommand{\ms}[1]{\ensuremath{\mathscr{#1}}}   
\newcommand{\IT}[2]{\ensuremath{IT}_{#1}( #2 )}      

\DeclareMathOperator{\conv}{conv}
\DeclareMathOperator{\Spec}{Spec}
\DeclareMathOperator{\supp}{supp}

\newtheorem{thm}{Theorem}
\newtheorem{Snum}[thm]{Proposition}

\newtheorem{Lem}[thm]{Lemma}
\newtheorem{Kor}[thm]{Corollary}

\theoremstyle{definition}

\newtheorem{defn}[thm]{Definition}

\theoremstyle{remark}
\newtheorem{Bem}{Remark}
\newtheorem{ex}{Example}

\author{E. Faber}
\author{D. B. Westra}

\begin{document}

 \title{Blowups in tame monomial ideals}
  \date{\today}
   \thanks{E.F. is supported by the grants Austrian Science Fund FWF P-18992 and
F-443 of the University of Vienna and D.B.W. is supported by the grant IK 1008-N
of the University of Vienna. This work resulted from elaborating the diploma
thesis of the first author.}

\begin{abstract}
We study blowups of affine $n$-space with center an arbitrary monomial
ideal and call monomial ideals that render smooth blowups {\it tame ideals}.
We give a combinatorial criterion to decide whether the blowup
is smooth and apply this criterion to discuss a smoothing procedure
proposed by Rosenberg, monomial building sets and permutohedra.
\end{abstract}

\maketitle

\tableofcontents

\section{Introduction}


In this paper we consider the question what happens when we choose a
nonregular closed subscheme as center of a
blowup and in particular we pose the question when the blowup of the
affine space $\A^n_K$ in a 
nonreduced subscheme defined by a monomial ideal is a smooth scheme. 
The base field $K$ will almost exclusively be taken to have characteristic zero and moreover, we sometimes need $K$ to be algebraically closed.
By a
blowup of an algebraic variety $X$ in a closed subscheme $Z\subset X$
we mean an algebraic  
variety $\widetilde{X}$ together with a proper birational map
$\pi:\widetilde{X}\to X$ that is an isomorphism outside $Z$. The
closed subscheme $Z$ is then called the center of the blowup. When the center $Z$ is defined by an ideal $I$, the blowup in $Z$ is also called the blowup in $I$. 

We give an example of a
surface where a resolution is obtained with one blowup in a nonregular
center (see \cite{hauser}):


\begin{ex}
Let $X \subseteq \A^3_\R$  be the surface defined by $x^2-y^3z^3=0$ (see
fig. \ref{fig:limao}).  The equation defining $X$ is invariant under the permutation of $y$ and $z$.
\begin{figure}[ht]
\begin{center}
\scalebox{0.3}{\includegraphics{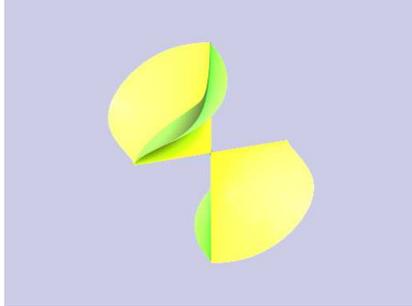}} \\
\caption{ \label{fig:limao} Image of the surface defined by $x^2-y^3z^3=0$.}
\end{center}
\end{figure}
The singular locus of $X$ is the coordinate cross $Z=V(x,yz)$.  The
traditional approach, as for example explained in 
\cite{villamayor,bierstone,BM}, is to first blow up in the origin or
in one of the axes. If we want to resolve $X$ by one blowup we have to
use a singular center.  
As one easily computes, the blowup of $\A^3_\R$ with center the
reduced ideal $(x,yz)$ is singular. Hence the blowup of $X$ is embedded in a
singular ambient scheme. One cannot speak of an ``embedded
resolution'' of $X$. However, the blowup in the nonradical ideal
$I=(x,y^2z,yz^2)(x,yz)(x,y)^2(x,z)^2$ resolves $X$ and the ambient space is
smooth. To achieve normal crossings one can modify the ideal $I$ or
perform further blowups. See \cite{Ha2} for a detailed description
of similar centers.
\end{ex}



In this paper we focus on monomial ideals in $K[x_1,\ldots,x_n]$ and
blowups of $\mathbb{A}^n$ in these ideals. In general, blowing up
$\mathbb{A}^n$ in a nonregular subscheme produces singularities. We
call a monomial ideal that renders a smooth blowup
$\widetilde{\mathbb{A}}^n$ a tame ideal.
The restriction to monomial ideals guarantees
that the blowup $\widetilde{\mathbb{A}}^n$ is a toric
variety, so that we can reduce the question of tameness to
combinatorics. The smooth affine toric varieties take a special simple
form; they are a direct product of an affine space and a
torus. Hence monomial ideals provide
a convenient testing ground to study blowups in nonregular
subschemes.

Theorem \ref{smoothness} gives a necessary and sufficient condition for a
monomial ideal to be tame; the criterion uses the structure of the
Newton polyhedron associated to the monomial ideal. We discuss several
criteria for blowups in products of monomial ideals to be tame.
We apply these criteria to three known constructions: $(i)$ Rosenberg
smoothing, $(ii)$ blowing up in building sets of arrangements of
linear subspaces and $(iii)$ permutohedral blowups, that is, blowups
in monomial centers that are invariant under any permutation of the
coordinates.

Rosenberg \cite{Ro} considers monomial ideals whose zero sets are unions of coordinate axes. These ideals are not tame. In \cite{Ro} two constructions are given to modify such an ideal $I$ such that the zero set is unchanged and the modified ideal is tame. The two constructions are intersecting and multiplying the ideal $I$ with a suitably chosen ideal.

A set of subspaces of a
vector space $V$ that is closed under taking sums is called an arrangement of linear subspaces in $V$. A building set of
an arrangement is a subset of the arrangement such that any element
$U$ of the arrangement can be written as a direct sum of elements of
the building set that are maximal with respect to inclusion in $U$.  
De Concini and Procesi \cite{DP} have shown that one can construct so-called
wonderful models of subspace arrangements by a sequence of blowups
with the elements of a building set as 
centers. MacPherson and Procesi \cite{MP} and Li \cite{LiLi}
generalized the construction of De Concini and Procesi to wonderful conical
compactifications and to arrangements of subvarieties of a smooth variety,
respectively. We consider the case
of linear subspaces that are defined by monomial coordinate ideals and
provide an elementary proof of the result of De Concini and Procesi in
this particular setting. 

The group $S_n$ acts in a natural way on $\mathbb{A}^n$ by permuting
the coordinates. We consider subschemes that are invariant under $S_n$
and can be written as unions of coordinate subspaces. In general, such
subschemes are nonregular. We show that blowing up $\mathbb{A}^n$ in a
certain class of monomial ideals that are invariant under $S_n$ and
whose zeroset is a union of coordinate subspaces results in a smooth
toric variety $\widetilde{\mathbb{A}}^n$. This class of ideals is
related to permutohedra (see e.g. \cite{Post}) and therefore these
ideals are called permutohedral ideals.

The outline of the paper is as follows. In Section \ref{SECsmooth} we
give a brief 
overview of blowups with the focus on monomial ideals, mainly to fix the
notation and the setting. In Section \ref{SECcriterion} 
we state and prove the smoothness criterion, which is then applied to discuss
the Rosenberg smoothing procedure in Section \ref{SECrosenberg}. Criteria for the tameness of products of
monomial ideals and in particular of products of coordinate ideals is studied in
Section \ref{SECproducts}.  We apply these criteria to building sets in Section
\ref{SECbuilding} and to permutohedral ideals in Section \ref{SECpermutation}.
\vskip 0.5cm
{\bf Acknowledgments.} The authors wish to thank Herwig Hauser for supervising
the diploma thesis of E.F. and many helpful discussions; Li Li, Claudio 
Procesi and Bernd Sturmfels for valuable discussions; Roc\'io Blanco, Clemens Bruschek,
Alexandra Fritz, Martin Fuchs and Peter Schodl for discussions and
proof reading.


\section{Blowups of $\A^n$ in monomial ideals} \label{SECsmooth}


We work with a smooth affine scheme $W \cong \A^n_K
\cong\Spec(K[\xin])$ over a field $K$. When the
field $K$ is irrelevant, we write $\A^n$ instead of
$\A^n_K$. For an ideal $I$ we write $V(I)$ for its zero set. To keep
the notation simple we write 
$K[{\bf x}]$ occasionally for $K[\xin]$ and for ${\bf a}=(a_1,\ldots,a_n)\in
\Z^n$ we use ${\bf x}^{\bf a}=x_{1}^{a_1}\cdots x_{n}^{a_n}$. Furthermore, the
standard basis vectors in $\R^n$ are denoted ${\bf e}_1,\ldots,{\bf
  e}_n$. Below we shortly discuss the construction of the blowup of
$\A^n$ to fix the notation and to provide the setting in which we
work. We also give some definitions that will be used frequently. For
more complete discussions of and introductions to blowups 
and resolution of singularities we refer to \cite{cut,Ha4,kollar,Li}.

A blowup of $\A^n$ is a scheme $\widetilde \A^n$ together with a
projection $\pi:  \widetilde \A^n \to \A^n$, the associated
blowup map. Each blowup is completely determined by a closed
subscheme $Z\subset \A^n$, which is called the center
of the blowup. Equivalently, the ideal $I$ that defines $Z$ determines the blowup completely.
The center $Z$ is the locus of points above which $\pi$ is not an isomorphism.
We also say that we blow up the associated coordinate ring $K[\xin]$ with center
$I$. The map $\pi: \widetilde \A^n \to \A^n$ is constructed as follows:
since $\A^n$ is Noetherian, $I$ is finitely generated, say, $I=(g_1,
\ldots, g_k) \subseteq K[\xin]$. The blowup
$\widetilde \A^n$ of $\A^n$ with center $Z=V(I)$ is defined as the
Zariski-closure of the graph of the map
\begin{equation}
\sigma: \A^n \backslash Z \to \Proj^{k-1}\,, ~~~p \mapsto (g_1(p):
\ldots: g_k(p))\,.
\end{equation}
Thus the blowup $\widetilde \A^n$ lives in $ \A^n \times
\Proj^{k-1}$. The projection $\pi: \widetilde \A^n \to \A^n$
on the first factor is the blowup map, and
the preimage $\pi^{-1}(Z)$ of the center $Z$ is called the exceptional
locus. The blowup $\widetilde \A^n$
can be covered by $k$ affine charts, each corresponding to a generator
$g_i$ of $I$. The coordinate ring of the
$i$th chart is
\begin{equation}
K[\xin, \frac{g_1}{g_i}, \ldots, \frac{g_k}{g_i}]\,.
\end{equation}

A smooth center defined by monomials is a coordinate subspace
$Z=V(x_1, \ldots, x_k)\subset \A^n$. The coordinate ring of the $i$th
affine chart of the blowup then reduces to
\begin{equation}
K[\frac{x_1}{x_i}, \ldots,
\frac{x_{i-1}}{x_i}, x_i, \frac{x_{i+1}}{x_i}, \ldots,
\frac{x_k}{x_i}, x_{k+1},\ldots, x_n]\,.
\end{equation}
The associated blowup map $\pi_{x_i}: \widetilde \A^n \to \A^n$ is
given in the $i$th chart by
\begin{equation}
x_j \mapsto \left\{ \begin{array}{l@{\quad}l} x_i x_j\,,& \text{ if }~ j
    \neq i ~\text{ and }~ j \leq k\,, \\  x_j\,, & \text{ otherwise. } \\
  \end{array}\right.
\end{equation}

More generally, let $\Id$ be any monomial ideal in $K[\xin]$
and let $\pi: \widetilde \A^n \to \A^n$ be the blowup of $\A^n$ with
center $\Id$. Then $\widetilde \A^n$ is a toric variety, which is not
necessarily normal and could be singular. Let $\{\vek{x}^\vek{a}:
\vek{a} \in A \}$, with $A$ a finite set in $\N^n$, be a set of
generators of $\Id$. The blowup $\widetilde \A^n$ of the affine space
will be covered with the affine toric varieties $U_\vek{a}$ given by
\begin{equation}
U_\vek{a}:= \Spec(K[\xin][\vek{x}^{\vek{a}'-\vek{a}}, \vek{a}' \in A
\backslash \vek{a}])\,.
\end{equation}
If $U_\vek{a}$ is smooth and $K=\C$ then $U_\vek{a}$ is isomorphic to
$\C^{n-k}\times (\C^*)^k$ for some $k$, see for example \cite[chapter
2]{Fu}, \cite{Ew}. We call $U_\vek{a}$ the $\vek{a}$-chart of
$\widetilde \A^n$.



\vskip 0.4cm

Using monomial ideals has the advantage that we can apply techniques from convex
geometry and toric geometry and that we can use combinatorial arguments. In the
paragraphs below we recall some basic notions of convex geometry in order to
introduce the notion of the ideal tangent cone and discuss some of its
properties. Most of our definitions can be found in standard
text books such as \cite{Fu,Ew,Schr,Od,St}.

We call a subset $C$ of $\R^n$ a polyhedral cone if there exist
vectors $\vek{v}_1, \ldots, \vek{v}_l \in \R^n$, such that $C= \{
\sum_{i=1}^l \lambda_i\vek{v}_i: \lambda_1, \dots, \lambda_l \in \R_+
\}$. We then say that $C$ is generated by $\vek{v}_1, \ldots,
\vek{v}_l$ and  write $C= {}_{\R_{+}}\langle \vek{v}_1, \ldots,
\vek{v}_l \rangle$. For the cone generated by the standard basis
vectors of $\R^n$ we write $\R^n_+$. A polyhedral cone is called
  rational if one can find generators $\vek{v}_i$ with integer
coordinates. We then call a generator $\vek{v}_i$ in $\Z^n$ primitive
if its coordinates are relatively prime. For any cone $C$ we define
the associated lattice cone to be $C \cap \Z^n$. If $C$ is a rational
polyhedral cone then $C \cap \Z^n$ is finitely generated over $\Z$ by
Gordan's lemma (see for example \cite{Fu}).
For any subset $N$ of $\R^n$ we write $\conv(N)$ for the convex hull
of $N$. The positive convex hull of a subset $M$ of $\R^n$ is the
convex hull of the Minkowski sum of $M$ with $\R_+^n$, written
$\conv(M+\R^n_+)$. A polytope is the convex hull of a finite set of points. The
Minkowski sum of two sets $P$ and $Q$ is the set consisting of all points $p+q$
where $p$ and $q$ run over $P$ and $Q$, respectively. A polyhedron is the
Minkowski sum of a polytope and a polyhedral cone. 

\begin{Bem} \label{Rmk:summe}
We make an observation that will be used frequently in the sequel. Let
${\bf a}$ and ${\bf b}$ be two different
points in the intersection of two polyhedra $P$ and $Q$. Then ${\bf
a}+{\bf b}$ is not a vertex of $P+Q$. Indeed, $2{\bf a}$ and $2{\bf
b}$ are in $P+Q$ and ${\bf a}+{\bf b} = \tfrac{1}{2}(2{\bf
a})+\tfrac{1}{2} (2{\bf b})$.
\end{Bem}

\begin{defn}
Let $P=\conv(M) +{}_{\R_+}\!\langle \vek{v}_1, \ldots, \vek{v}_l
\rangle$ with $M \subseteq \R^n$ finite, be a unbounded polyhedron.
If for a vertex $\vek{m}$ of $\conv(M)$ the ray $\vek{m}+ \R_+
\vek{v}_i$ is an edge of $P$   we call $\vek{m}+ \R_+ \vek{v}_i$ the
 infinitely far vertex of $P$ in direction $\vek{v}_i$.
\end{defn}

The support of a monomial ideal $\Id$ is the set of all exponent
vectors of monomials in $\Id$,
\begin{equation}
 \supp( \Id)= \{ \vek{a} \in \N^n: \vek{x}^\vek{a} \in \Id \}\,.
\end{equation}
Its convex hull $\conv( \supp(\Id))$ is called the Newton polyhedron
of $\Id$, which is denoted by $N(\Id)$. Equivalently, for a
finite set of generators $\vek{x}^\vek{a}$, $\vek{a} \in A$, of $\Id$
the Newton polyhedron is defined as the positive convex hull $\conv(A +
\R^n_+)$. It is important to note that $N(\Id)\cap \Z^n\supset \supp
(\Id)$ but that the inclusion can be proper. For example, the ideals
$(x,y)^2$ and $(x^2,y^2)$ in $K[x,y]$ have the same Newton polyhedron,
but not the same support.

\begin{Bem}
A monomial ideal $I$ is integrally closed if and only if all lattice
points of the Newton polyhedron of $I$ are in the support of
$I$, see for example \cite{eisenbud,teissier88}. It is a well-known
theorem from Zariski \cite{ZariskiSamuel} that any integrally closed ideal is
complete. Furthermore, Zariski proves that complete ideals in $K[x,y]$
can be written as a product of simple complete ideals, which then
correspond to the exceptional divisors of blowing up in the simple
complete ideals; see for example
\cite{ZariskiSamuel,cutkosky1,cutkosky2,cutkosky3,spivakovsky} for
more on complete ideals and a full list of references.
\end{Bem}

The following definition is of crucial importance for the sequel:

\begin{defn}
For a set of vectors $\vek{v}_1, \ldots, \vek{v}_l \in \R^n$ the set
$\{ \lambda_1 \vek{v}_1 + \ldots + \lambda_l \vek{v}_l: \lambda_i \in
\N \}= {}_\mathbb{N}\langle \vek{v}_i: 1\leq i \leq l \rangle$ is
called the $\N$-span of the $\vek{v}_i$. For two $\N$-spans $C_1={}_\N\langle
\vek{v}_1,\ldots,\vek{v}_s\rangle$ and $C_2={}_\N\langle
\vek{w}_1,\ldots,\vek{w}_r\rangle$ we define the sum $C_1+C_2$ to be
the $\N$-span ${}_\N\langle \vek{v}_1,\ldots,\vek{a}_s,
\vek{w}_1,\ldots,\vek{w}_r\rangle$; any element of $C_1+C_2$ is a sum
$x+y$ with $x\in C_1$ and $y\in C_2$.
We write $\Sigma_n$ for ${}_\N\langle \vek{e}_1,\ldots,{\bf e}_n\rangle$. For a
monomial ideal $\Id= (
\vek{x}^\vek{a}: \vek{a} \in A )$, $A \subseteq \N^n$ we define
\begin{equation}
\IT{\vek{a}}{\Id}:={}_\mathbb{N}\langle  \vek{a}'- \vek{a}: \vek{a}'
\neq \vek{a} \in A \rangle + \Sigma_n \,,
\end{equation}
the ideal tangent cone of the monomial $\vek{x}^\vek{a}$. We call
$\IT{\vek{a}}{\Id}$ pointed if $\IT{\vek{a}}{\Id}
\cap (-\IT{\vek{a}}{\Id})= \{0 \}$. A minimal set of generators
of $\IT{\vek{a}}{\Id}$ is a finite set $S$ of vectors in $IT_{{\bf a}}(\Id)$
that generate $\IT{\vek{a}}{\Id}$ and no element in $S$ is an
$\N$-linear combination of the other elements in $S$. We call an element of
the minimal generating system a minimal generator of
$\IT{\vek{a}}{\Id}$.
\end{defn}

We remark that the ideal tangent cone is in general not a polyhedral
lattice cone. For a point ${\bf a}$ in the support of $\Id$, we define the real
tangent cone in ${\bf a}$ to be the polyhedral cone $T_{{\bf a}}(\Id)$ generated
by all vectors ${\bf p} - {\bf a}$ where ${\bf p}$ lies in $N(\Id)$. The real
tangent cone $T_{{\bf a}}(\Id)\cap \Z^n$ contains $IT_{{\bf a}}(\Id)$, and the
inclusion can be proper.

\begin{Lem}
Let ${}_\N\langle\vek{v}_1, \ldots, \vek{v}_l \rangle= \IT{\vek{a}}{\Id}$ be
the ideal tangent cone of $N$ in a point $\vek{a}$. If the ideal
tangent cone is pointed it has a unique minimal set of generators.
\end{Lem}

\begin{proof}
Assume $\IT{\vek{a}}{\Id}$ has two different minimal sets of
generators $\vek{v}_1, \ldots, \vek{v}_l$ and $\vek{w}_1, \ldots,
\vek{w}_m$, and assume $\vek{v}_1 \not \in \{ \vek{w}_1, \ldots,
\vek{w}_m \}$. Then there are $\alpha_{ji}\in \N$ and $\beta_i\in
\N$ such that
\begin{equation}
\vek{w}_j = \sum_i \alpha_{ji}\vek{v}_i\,,\quad \vek{v}_1 =
\sum_j\beta_j\vek{w}_j\,.
\end{equation}
Hence we have $\vek{v}_1= \sum_{i,j}\beta_j \alpha_{ji}\vek{v}_i$
The coefficient $\sum_j \alpha_{j1} \beta_j$ must be greater or equal
$1$ because otherwise $\vek{v}_1 \in {}_\N \langle\vek{v}_2, \ldots,
  \vek{v}_l \rangle$, which contradicts that $\vek{v}_1, \ldots,
\vek{v}_l$ form a minimal system of generators. If $\sum_j \alpha_{j1}
\beta_j > 1$, $0$ would be a nontrivial $\N$-linear combination of
the $\vek{v}_i$, contradicting $\IT{\vek{a}}{\Id}$ to be
pointed. Hence the only possibility is $\alpha_{1k}=1$ and $\beta_k=1$
for some $k$ and zero otherwise. But then $ \vek{v}_1= \vek{w}_k$,
which contradicts that $\vek{v}_1\notin \left\{ \vek{w}_1, \ldots,
\vek{w}_m \right\}$.
\end{proof}

\begin{Lem}
Let $N$ be the Newton polyhedron of a monomial ideal $\ms{I}$ in
$K[x_1,\ldots, x_n]$ and ${\bf a}$ be in $\supp(\Id)$. Then the ideal
tangent cone $IT_{{\bf a}}(\ms{I})$ is pointed if and only if ${\bf
a}$ is a vertex.
\end{Lem}

\begin{proof}
First assume ${\bf a}$ is a vertex.
Suppose $IT_{{\bf a}}(\ms{I}) = {}_{\N}\langle {\bf
  a}_1,\ldots,{\bf a}_m \rangle$, with ${\bf a}_i \in
\Z^n$. The vectors ${\bf a}_i$ are vectors of the form ${\bf
  b}_i - {\bf a}$ where the ${\bf b}_i$ are some (rather special)
points of $N$, not equal to ${\bf a}$ and the standard basis vectors
$\vek{e}_1, \ldots, \vek{e}_n$. If
$IT_{{\bf a}}(\ms{I})$ is not pointed, then we can write $0= \sum_i
\nu_i{\bf a}_i$ with $\nu_i\in \N$
not all zero. But then ${\bf a} = \sum_i \mu_i {\bf b}_i$ with $\mu_i
= \nu_i / \sum_i \nu_i$ is a convex linear combination of other
points in $N$ and ${\bf a}$ is not a vertex.

Conversely, if ${\bf a}$ is not a vertex, then ${\bf a}$ is a
$\mathbb{Q}$-linear combination of some vertices ${\bf b}_i$ of $N$:
${\bf a}=\sum_i \lambda_i{\bf b}_i$ with $\lambda_i\in\mathbb{Q}$,
positive and
$\sum_i\lambda_i =1$. Multiplying with a common denominator $d\in\N$
we get: $d{\bf a} - \sum_i \mu_i {\bf b}_i=0$ and $\sum_i
\mu_i=d$. The ${\bf a}-{\bf b}_i$ generate the tangent cone and
$\sum_i \mu_i({\bf a} - \mu_i {\bf b}_i)$ is a nontrivial $\N$-linear
combination representing zero and thus $IT_{{\bf
a}}(\ms{I})$ is not pointed.
\end{proof}

By definition of the support of an ideal $\Id=({\bf x}^{\bf a}:{\bf
  a}\in A)$, $\supp(\Id)$ consists
of all ${\bf c}\in\mathbb{Z}^n$ such that ${\bf x}^{\bf c}\in
\Id$. But then ${\bf c}$ must be of the form ${\bf a} + m_1{\bf e}_1+
\ldots + m_n {\bf e}_n$ for some ${\bf a}\in A$ and
some nonnegative integers $m_1,\ldots,m_n$.
It follows that for a fixed ${\bf b}\in A$ we have
\begin{equation}
{}_{\mathbb{N}} \langle {\bf a} - {\bf b} : {\bf a}\in \supp(\Id) \rangle =
{}_{\mathbb{N}} \langle {\bf a}' - {\bf b} : {\bf a}'\in A \rangle +
\Sigma_n\,.
\end{equation}

\begin{defn}
Let $\Id \subset K[\xin]$ be a monomial ideal. We call a pointed ideal
tangent cone $\IT{\vek{a}}{\Id}$
simplicial if the minimal system of generators of $\IT{\vek{a}}{\Id}$
consists of exactly $n$ vectors.
\end{defn}

Since the standard basis vectors are contained in each ideal tangent
cone, the following lemma is obvious:

\begin{Lem}
Let $\Id=(\vek{x}^\vek{a}: \vek{a} \in A)$ be an arbitrary monomial
ideal and $N=N(\Id)$ the associated Newton polyhedron. If the ideal
tangent cone of an $\vek{a} \in N$ is simplicial then the set of
minimal generators form a basis of $\Z^n$.
\end{Lem}

Let $\IT{\vek{a}}{\Id}= {}_\N\langle \vek{v}_1, \ldots, \vek{v}_l \rangle$
with $l \geq n$ and let  $K[\IT{\vek{a}}{\Id}]=K[\vek{x}^{\vek{v}_1},
\ldots, \vek{x}^{\vek{v}_l}]$ be the monomial algebra generated by
$\IT{\vek{a}}{\Id}$. It can be seen easily that
$K[\IT{\vek{a}}{\Id}]$ is the coordinate ring of the $\vek{a}$-chart
$U_\vek{a}$ of $\widetilde \A^n$: The exponents $\vek{a}'- \vek{a}$,
with $\vek{a}' \in A \backslash \vek{a}$, of the generators ${\bf
  x}^{{\bf a}'-{\bf a}}$ of the coordinate ring $K[{\bf x},
\vek{x}^{\vek{a}'- \vek{a}}: \vek{a}' \in A \backslash \vek{a}]$ of
the $\vek{a}$-chart of the blowup correspond in the Newton polyhedron $N$ to
the vectors pointing from $\vek{a}$ to $\vek{a}'$. The generators $\xin$
of $K[{\bf x}, \vek{x}^{\vek{a}'- \vek{a}}: \vek{a}' \in A \backslash
\vek{a}]$ correspond in the Newton polyhedron $N$ to the $n$
standard basis vectors. These generators make up the
infinitely far vertices $\vek{a} + \R_+ \vek{e}_i$. We have the inclusion
$K[\IT{\vek{a}}{\Id}] \subseteq K[\vek{x}, \vek{x}^{-1}]$. Hence $U_{{\bf a}}$
contains a torus $(K^*)^n$ as a dense subset.

Now consider the monoid homomorphism
$\pi: \N^l \to \Z^n$, $\vek{b} \mapsto \sum_{i=1}^l b_i
\vek{v}_i$. The image of $\pi$ is $\IT{\vek{a}}{\Id}$ and we get a
homomorphism of monomial algebras
\begin{equation}
\hat \pi: K[t_1, \ldots, t_l] \to K[\vek{x},\vek{x}^{-1}]\,, \quad t_i \mapsto
\vek{x}^{\vek{v}_i}\,.
\end{equation}
This construction yields the explicit description of the
$\vek{a}$-chart of $\widetilde \A^n$ as a toric variety. The following
lemma can be found in standard textbooks, like \cite{Fu, Ew}.

\begin{Lem}\label{lem.Ewald}
Let  ${\bf
  v}_1,\ldots , {\bf v}_l$ in $\Z^n$. The kernel of the map
$K[t_1,\ldots,t_l]\to K[{\bf x}^{{\bf v}_1},\ldots, {\bf x}^{{\bf
    v}_l}]$ defined by $t_i \mapsto {\bf x}^{{\bf v}_i}$ is generated
by binomials of the form ${\bf t}^{\alpha} - {\bf t}^{\beta}$ for
some $\alpha,\beta\in \N^m$.
\end{Lem}

\begin{proof}
Let $f\in K[t_1,\ldots,t_l]$ be in the kernel. We expand $f$ as a sum
of monomials $f= \sum_{ \alpha}c_{ \alpha} {\bf t}^{\alpha}$, where
$\alpha = (\alpha_1,\ldots,\alpha_l)$ is a multi-index. Since $f$ maps
to zero, we get $\sum_{ \alpha}c_{ \alpha} {\bf x}^{ \alpha \cdot {\bf
    v}}=0$, where $\alpha \cdot {\bf v} = \sum_{i=1}^{l}\alpha_i {\bf
  v}_i \in \Z^n$. Hence for all ${\bf w}\in\Z^n$
\begin{equation}
\sum_{ \alpha;\,  \alpha \cdot {\bf v} = {\bf w}} c_{ \alpha} = 0\,.
\end{equation}
Thus if for some $\alpha$ the
coefficient $c_{\alpha}$ is nonzero, there is another $\beta$ with
$\alpha\cdot {\bf v} = \beta \cdot {\bf v}$ and $c_{\beta}\neq
0$. Then we consider $f' = f - c_\alpha ({\bf t}^\alpha - {\bf
 t}^\beta)$, which has less monomial terms than $f$. Hence the proof
is completed by induction on the number of monomial terms.
\end{proof}


\section{A smoothness criterion}\label{SECcriterion}


The main result of this section is Theorem \ref{smoothness}, which
contains a smoothness criterion and characterizes
the ideal tangent cones of the smooth affine open charts $U_{{\bf
    a}}$. The result is the starting point for our further explorations.

The open affine charts $U_{{\bf a}}$ of the blowup
$\widetilde{\A}^n$ are by their construction toric
varieties. Requiring that $U_{{\bf a}}$ is smooth singles out a unique
affine toric variety.

\begin{Lem}\label{smoothtoric}
Let $\ms{I}=( {\bf x}^{\bf a}: {\bf a}\in A)$ be a
monomial ideal in $K[{\bf x}]$, $N$ the associated Newton polyhedron
and $\pi:  \widetilde{\A}^n \to \A^n$ the blowup of
$\A$ in the center $\ms{I}$. If ${\bf a}$ is a vertex of $N$
such that the ${\bf a}$-chart $U_{{\bf a}}$ is smooth, then the ${\bf
  a}$-chart is isomorphic to $\A^n$.
\end{Lem}
\begin{proof}
By construction $U_{{\bf a}}$ is an $n$-dimensional affine toric
variety, and since it is smooth, we have $U_{{\bf a}} \cong
K^{n-k} \times (K^*)^{k}$ for some $k$, see e.g.
\cite{Fu}. But if $k$ is nonzero the ideal tangent cone of the
vertex ${\bf a}$ is not pointed.
\end{proof}

\begin{thm}\label{smoothness}
Let $\ms{I}=( {\bf x}^{\bf a}: {\bf a}\in A)$ be a
monomial ideal in $K[{\bf x}]$, $N$ the associated Newton polyhedron
and $\pi: \widetilde{\A}^n \to \A^n$ the blowup of
$\A^n$ in the center $\ms{I}$. Then we have:
\begin{itemize}
\item[(i)] When ${\bf a}\in A$ is not a vertex of $N$, then the ${\bf
    a}$-chart is already covered by the affine charts of
  $\widetilde{\A}^n$ corresponding to the vertices of $N$.
\item[(ii)] When ${\bf a}$ is a vertex of $N$, then the ${\bf
    a}$-chart is smooth if and only if the ideal tangent cone
  $IT_{{\bf a}}(\ms{I})$ is simplicial.
\item[(iii)] When ${\bf a}$ is a vertex of $N$ and the ideal tangent
  cone is simplicial, then each minimal generator of $IT_{{\bf
      a}}(\ms{I})$ is primitive and $IT_{{\bf a}}(\ms{I}) = T_{{\bf
      a}}(N)\cap \Z^n$.
\item[(iv)] When ${\bf a}$ is a point of $\supp(I)$ such
  that all its neighboring lattice points are in $\supp(\Id)$,
  then the ${\bf a}$-chart is isomorphic to $(K^*)^n$.
\end{itemize}
\end{thm}

\begin{proof}
$(i)$: If ${\bf a}\in N$ is not a vertex, ${\bf a}$ is
contained in the convex hull of some vertices ${\bf a}_1,\ldots,{\bf
  a}_m$ of $N$. We thus can write ${\bf a} = \sum_{i=1}^{m} \lambda_i
{\bf a}_i$ where the $\lambda_i$ are nonzero positive rational numbers
with $\sum_{i=1}^{m}\lambda_i = 1$. The vectors ${\bf a}_i - {\bf a}$
are in the ideal tangent cone and
\begin{equation}
\sum_{i=1}^{m} \lambda_i ({\bf a}_i - {\bf a}) = 0\,.
\end{equation}
Thus by multiplying with a common denominator of the
$\lambda_i$ we can write $0 = \sum_{i=1}^{m} \alpha_i ({\bf a}_i -
{\bf a})$ with $\alpha_i \in \N$ and $\alpha_i \geq 1$ for
all $i$. Thus $IT_{{\bf a}}(\ms{I})$ is not pointed and for any $j$ we
have
\begin{equation}
{\bf a} - {\bf a}_j = \sum_{i\neq j}  \alpha_i ({\bf a}_i -{\bf a} ) +
(\alpha_j -1)({\bf a}_j - {\bf a}) \,\in \, IT_{{\bf a}}(\ms{I})\,.
\end{equation}
But then it follows that $IT_{{\bf a_i}}(\ms{I}) \subset IT_{{\bf
    a}}(\ms{I}) $ for all ${\bf a}_i$, since for any ${\bf a}'\neq
{\bf a}_i$ we have
\begin{equation}
{\bf a}' - {\bf a}_i = ({\bf a}' - {\bf a}) + ({\bf a} - {\bf a}_i)
\,\in\,IT_{{\bf a}}(\ms{I}) \,.
\end{equation}
Therefore $K[IT_{{\bf a}_i}(\ms{I})] \subset K[IT_{{\bf a}}(\ms{I})]$
and thus ${\rm Spec} (K[IT_{{\bf a}}(\ms{I})]) \subset {\rm
  Spec}(K[IT_{{\bf a}_i}(\ms{I})])$.

$(ii)$: The ideal tangent cone $IT_{{\bf a}}(\ms{I})$ is
generated by all standard basis vectors ${\bf e}_i$, for $1\leq i \leq
n$ of $\Z^n$ and the vectors ${\bf a}' - {\bf a}$ where ${\bf
  a}'$ runs over all vertices of $N$, except ${\bf a}$. Together we
denote these generators by ${\bf b}_i$, $1\leq i \leq M$. The ideal
tangent cone of ${\bf a}$ has a unique minimal generating set
$S$. Since we can try to eliminate any of the ${\bf b}_i$ to get a
minimal generating set, we have $S\subset \left\{ {\bf b}_1,\ldots,
  {\bf b}_M\right\}$.

First assume $U_{{\bf a}}$ is smooth. By Lemma \ref{smoothtoric} we
know that $U_{{\bf a}}$ is isomorphic to $\A^n$ and thus
$K[IT_{{\bf a}}(\ms{I})] \cong K[y_1,\ldots,y_n]$. Hence $S$ contains
exactly $n$ elements and we may assume $S= \left\{ {\bf b}_1,\ldots,
  {\bf b}_n\right\}$. Any ${\bf b}_i\in S$ can be expressed as an
$\N$-linear combination of the ${\bf b}_k$ with $1\leq k\leq n$. In
particular, the $\N$-span of the set $S$ contains the basis
vectors ${\bf e}_i$ of $\Z^n$. Hence $S$ is a basis for
$\Z^n$ and thus $IT_{{\bf a}}(\ms{I}) $ is simplicial.

Conversely, assume the set $S$ contains precisely $n$ linearly
independent elements. Then $S$ is a basis for
$\Z^n$. The coordinate ring of $U_{{\bf a}}$ is $K[{\bf
  x}^{{\bf b}_i}\,; 1\leq i\leq n]$ and we have a map
$K[z_1,\ldots,z_n]\to K[{\bf x}^{{\bf b}_i}\,; 1\leq i\leq M]$ sending
$z_i $ to ${\bf x}^{{\bf b}_i}$. This map is clearly surjective. By
Lemma \ref{lem.Ewald} we know that the kernel is generated by
binomials. Hence suppose $f={\bf z}^{{\bf w}} - {\bf z}^{{\bf u}} $ is
a generator of the kernel, then we must have $\sum_{i=1}^{n} w_i {\bf
  b}_i - \sum_{i=1}^{n} u_i {\bf b}_i=0$. Since the ${\bf b}_i$ are a
basis, we must have $u_i = w_i$ and thus $f=0$. Therefore the
coordinate ring of $U_{{\bf a}}$ is $K[z_1,\ldots,z_n]$ and $U_{{\bf
    a}} \cong {K}^n$.

$(iii)$: Suppose that $\frac{1}{r}{\bf b}_1$ is primitive
for some positive integer $r\geq 1$. Then the vectors $\frac{1}{r}{\bf
  b}_1,{\bf b}_2,\ldots,{\bf b}_n$ constitute a basis of
$\Z^n$. However the matrix that relates this basis to the
basis $S$ has determinant $ 1/r$. Hence we must have
$r=1$. Furthermore, we clearly have $IT_{{\bf a}}(\ms{I})\subset
T_{{\bf a}}(N)\cap \Z^n$. Suppose $p\in T_{{\bf a}}(N)\cap
\Z^n$. Then $p$ is an ${\rm I\!R}_{\geq 0}$-linear
combination and a $\Z$-linear combination of the vectors
${\bf b}_i$ in $S$. Since the ${\bf b}_i$ are a basis, the
coefficients in the expansions coincide and are thus in $\N =
\Z\cap {\rm I\!R}_{\geq 0}$. Thus $p\in IT_{{\bf
    a}}(\ms{I})$.

\noindent $(iv)$: When ${\bf a}$ is in each lattice direction surrounded by points of $\supp(\Id)$, then the ideal
tangent cone of ${\bf a}$ is all of $\Z^n$. Thus the
coordinate ring of $U_{{\bf a}}$ is
$K[z_1,\ldots,z_n,z_{1}^{-1},\ldots,z_{n}^{-1}]$, which proves the
claim.
\end{proof}

\begin{defn}
We call a monomial ideal $\Id ( \vek{x}^\vek{a}: \vek{a} \in A)$ in $K[{\bf x}]$
tame if the blowup $\widetilde \A^n$ of $\A^n$ is smooth. The corresponding
Newton polyhedron $N(\Id)$ is called tame if $\Id$ is tame.
\end{defn}


\section{Blowup in the Rosenberg ideal}  \label{SECrosenberg}


In \cite{Ro} Rosenberg constructs nonreduced monomial ideals $\ms{R}$
 in $K[{\bf x}]$ whose zero set is a union of coordinate axes in $\A^n$
and such that the blowup of $\A^n$ with center $\ms{R}$ is smooth. The
constructed ideals are invariant under any permutation of the coordinates of
$\A^n$ that leaves the zero set invariant. Rosenberg also generalizes the
construction to closed subschemes $Z$ of $\A^n$ that are unions of coordinate
subspaces of dimension $1 \leq r < n$. The idea of the construction is as
follows: Consider the reduced ideal $\Id$ of $Z$. The blowup of $\A^n$ with
center $\Id$ will in general be singular; from some vertices in the Newton
polyhedron $N(\Id)$ emanate too many edges, so that the associated ideal tangent
cones are not simplicial. One can transform $\Id$ by multiplication or
intersection with another monomial ideal such that the blowup becomes smooth and
the radical ideal $\Id$ is unchanged. We study the effect on $N(\Id)$ of this
method. For a sheaf-theoretic interpretation and further details also see
\cite{Ro}.

Let $1 < s < n$ and let $\Id$ be the reduced ideal of the first $s$ coordinate axes;
\begin{equation}\label{eqn.rosen}
\Id = \bigcap_{i=1}^s (x_1, \ldots, \hat x_i, \ldots, x_n)\,.
\end{equation}
The ideal $\Id$ is generated by the monomials $x_ix_j$ with $1 \leq i < j \leq
s$ and the monomials $x_i$ with $s < i$. We now briefly show that the blowup
with center $\Id$ is singular: Consider the $x_n$-chart. Here 
$u_{jk}=\frac{x_jx_k}{x_n}$  with $j \neq k \leq s$ are minimal generators.
These generators satisfy relations of the
form
\begin{equation}
 u_{jk} \cdot x_n - x_j x_k\,.
\end{equation}
The chart expression is then
\begin{equation}
\Spec \Bigl( K [ {\bf x}, u_{jk}, \frac{x_l}{x_n}: l > s ]
/ ( u_{jk}x_n - x_j x_k : j < k \leq s ) \Bigr)\,,
\end{equation}
which is clearly singular.

The idea to smooth $N(\Id)$ is to ``pull apart'' a vertex $\vek{v}$ of
$N(\Id)$, from which more than $n$ edges emanate. Each new vertex should inherit
some edges from $\vek{v}$ but the corresponding ideal tangent cone should have
exactly $n$ minimal generators. Rosenberg proposes two possibilities to smooth
nonsimplicial ideal tangent cones. The
first method consists in slicing the distracting vertex off $N(\Id)$, see
fig.\ref{fig:newtonchop}. This slicing corresponds to intersecting $\Id$ with
another ideal. The second method is to stretch the Newton polyhedron, that is,
one substitutes the polyhedron $N(\Id)$ with the Minkowski sum of $N(\Id)$ with
another suitable Newton polyhedron, see fig. \ref{fig:newtonsum}. This
stretching corresponds to multiplying $\Id$ with another ideal.

\begin{figure}[!h]
\begin{flushleft}
\begin{minipage}{0.4 \textwidth}\centering
 \scalebox{0.4}{\includegraphics{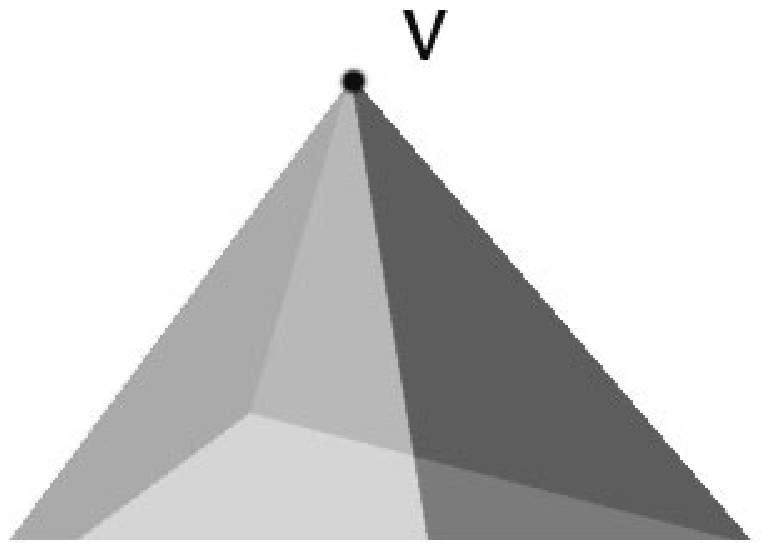}}

\end{minipage}%
\begin{minipage}{0.1 \textwidth}\centering
$\longrightarrow$
\end{minipage}%
\begin{minipage}{0.4 \textwidth} \centering
\scalebox{0.3}{\includegraphics{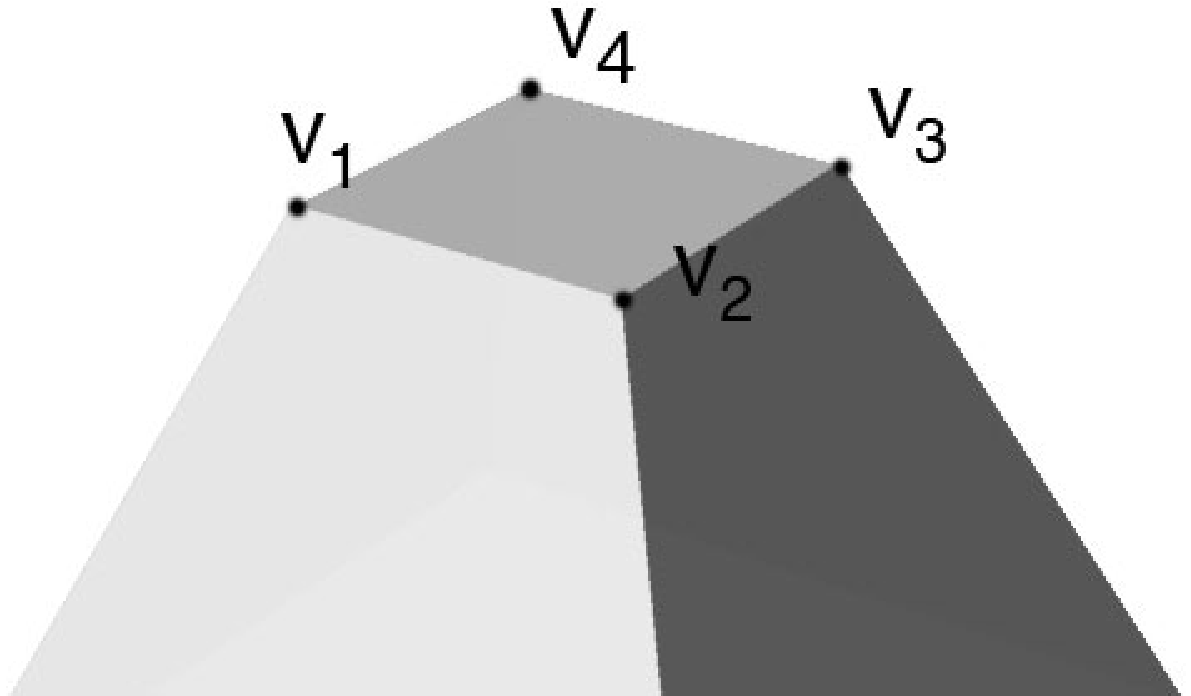}}

\end{minipage}
\end{flushleft}
\caption{ \label{fig:newtonchop} Chopping off a vertex with too many edges.}
\end{figure}

\begin{figure}[!h]
\begin{flushleft}
\begin{minipage}{0.4 \textwidth}\centering
 \scalebox{0.4}{\includegraphics{ecke1.eps}}

\end{minipage}%
\begin{minipage}{0.1 \textwidth}\centering
$\longrightarrow$
\end{minipage}%
\begin{minipage}{0.4 \textwidth} \centering
\scalebox{0.3}{\includegraphics{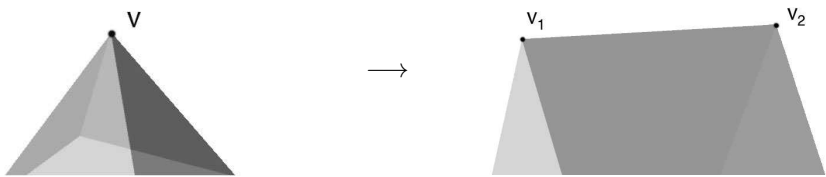}}

\end{minipage}
\end{flushleft}
\caption{ \label{fig:newtonsum}  The sum of two Newton polyhedra. }
\end{figure}

To leave the zero set unchanged, the radical ideal of $\Id$ has to be preserved.
For both smoothing methods, we use a power of the
maximal ideal $\mathfrak{m}$, where $\mathfrak{m}=(\xin)$
is the reduced ideal of the origin of $\A^n$. The choice of $\mathfrak{m}$ as
smoothing ideal is in no way unique! We could choose any ideal $\ms{J}$ such
that the blow up of $\A^n$ with center $\Id \cap \Jd$ or $\Id \cdot \Jd$ is
smooth and such that $\sqrt{\Id \cap \Jd}=\sqrt{\Id \cdot \Jd} =\Id$. But using
$\mathfrak{m}$ ensures the zero set is unchanged as we have $\mathfrak{m}\supset
\Id$ from which it follows that
\begin{equation}
\sqrt{ \Id \cap \mathfrak{m}^k}= \Id \cap \mathfrak{m} =
\Id\,.
\end{equation}

Since $\Id \cap \mathfrak{m} = \Id$, we consider the next option:
$\Id \cap \mathfrak{m}^2$. The associated Newton polyhedron still has
vertices with nonsimplicial ideal tangent cones (cf. fig.
\ref{fig:rosenbideal}, left: the Newton polyhedron of $\Id \cap
\mathfrak{m}^2$ for $n=3, s=2$). We can smooth these vertices by
further intersecting with $\mathfrak{m}^3$ or by multiplying $\Id\cap
\mathfrak{m}^2$ with $\mathfrak{m}$. As $\mathfrak{m}(\Id \cap
\mathfrak{m}^2) = \Id\cap \mathfrak{m}^3$ both methods yield the same result.

\begin{figure}[!h]
 \begin{minipage}[c]{0.5\textwidth}
\begin{center}
\phantom{aa}
 \scalebox{0.3}{\includegraphics{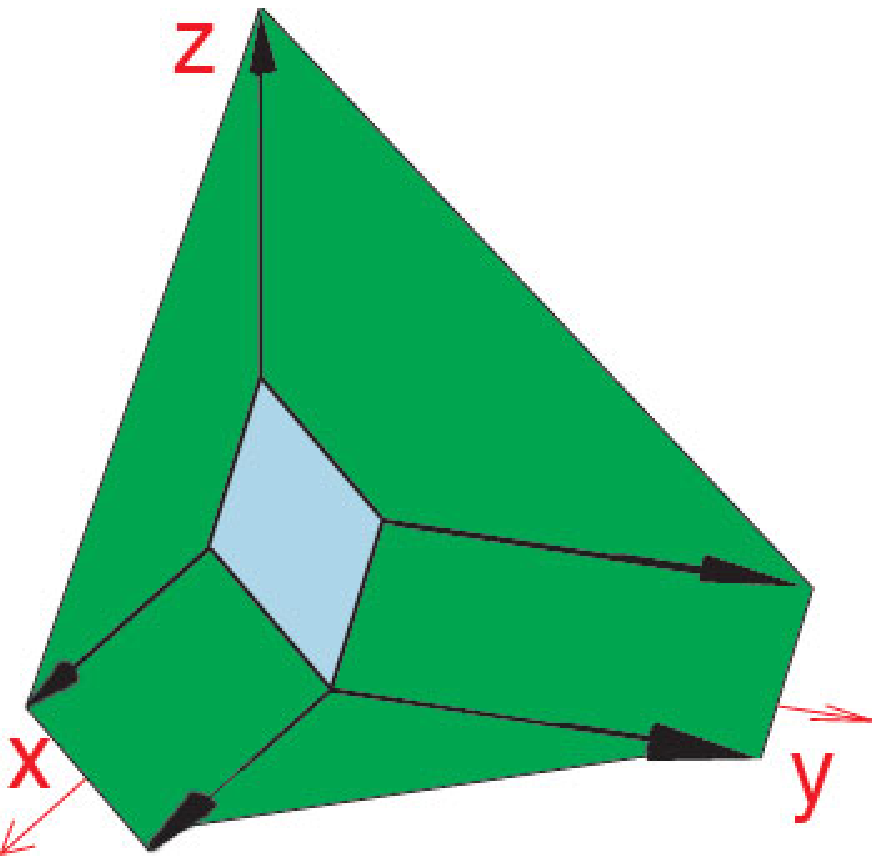}} \\
 \end{center}

\end{minipage}%
\begin{minipage}[c]{0.5\textwidth}

 \begin{center}
 \phantom{aa}
\scalebox{0.3}{\includegraphics{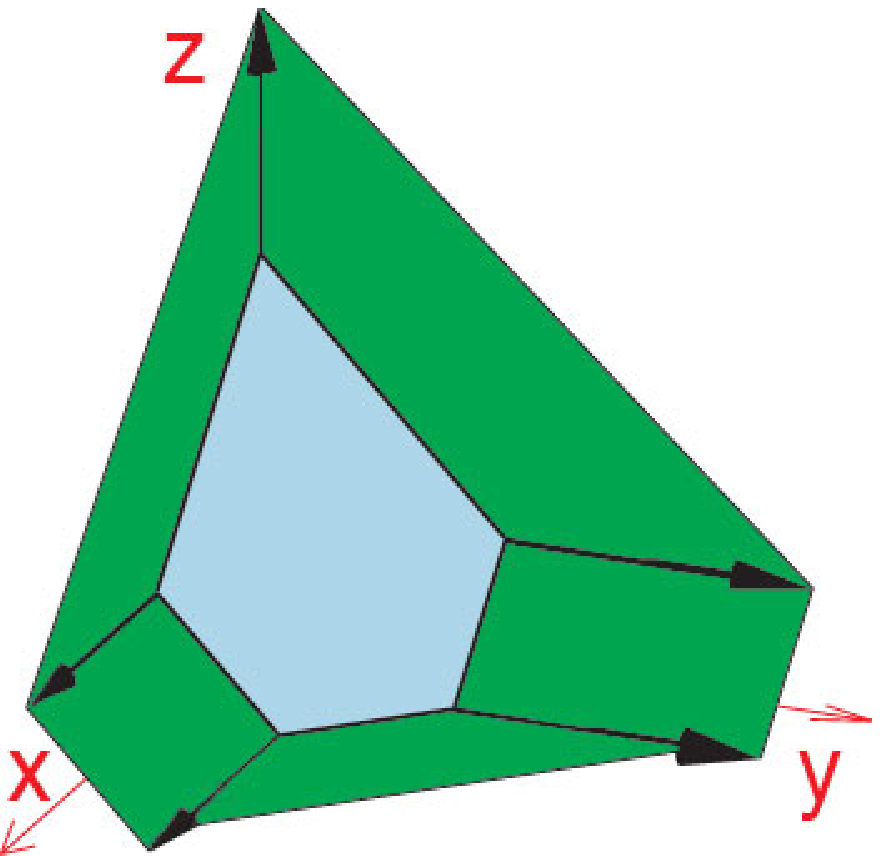}} \\
\end{center}
\end{minipage}
\caption{ \label{fig:rosenbideal} The Newton polyhedra of  $\Id \cap
\mathfrak{m}^2$ (left) and $\ms{R}$ (right). }
\end{figure}

We define $\ms{R}:= \Id \cap \mathfrak{m}^3$ and observe that the ideal $\ms{R}$
is invariant under each permutation of the coordinates of $\A^n$ that leaves
$\Id$ invariant.

\begin{Snum}
Let $\Id$ be the reduced ideal of the first $s$ coordinate axes of $\A^n$
defined in eqn.(\ref{eqn.rosen}). Then the ideal $\ms{R}= \Id \cap
\mathfrak{m}^3$ is tame. 
\end{Snum}

 \begin{proof}
We show the smoothness of $\widetilde \A^n$ with Theorem \ref{smoothness}. The Newton polyhedron $N(\ms{R})$ has the set of vertices
\begin{equation}
\{ 2 \vek{e}_i + \vek{e}_j : 1 \leq i \leq s, 1 \leq j \leq n, i
\neq j \} \cup \{ 3 \vek{e}_i: s < i \leq n \}\,.
\end{equation} 
It can easily be computed that the remaining generators of $\ms{R}$
correspond to interior points of $N(\ms{R})$ or lie in faces of the
Newton polyhedron. For example, for $2\vek{e}_j + \vek{e}_i$ with $1
\leq i \leq s < j \leq n$ we have 
\begin{equation}
2\vek{e}_j + \vek{e}_i = \frac{1}{2} (2\vek{e}_i + \vek{e}_j) +
\frac{1}{2}(3 \vek{e}_j)\,. 
\end{equation}
Consider the ideal tangent cone $\IT{2 \vek{e}_i + \vek{e}_j}{\ms{R}}$
for $1 \leq i \leq s$ and $1\leq  j \leq n$. The associated edge
vectors are the 
$\vek{e}_i$ for the infinitely far vertices, $\vek{e}_j - \vek{e}_i$
and $\vek{e}_k - \vek{e}_j$ with $k \not \in \{ i, j \}$ for the
adjacent vertices. There are exactly $n$ minimal generators and the
ideal tangent cone is simplicial.

For a vertex $3\vek{e}_i$ with $s<i\leq n$ the ideal tangent cone is generated
by $\vek{e}_i$ (infinitely far vertex) and the vectors $\vek{e}_k -\vek{e}_i$
for $k \neq i$ (adjacent vertices). Thus $\IT{3
\vek{e}_i}{\ms{R}}$ is simplicial. The blowup of $\A^n$ with center $\ms{R}$ is
therefore smooth.
\end{proof}


\section{Blowups in products of ideals}  \label{SECproducts}


In this section we apply the smoothness criterion \ref{smoothness} to
investigate which products of monomial ideals are tame. Li \cite{LiLi}
has shown that a sequence of blowups in ideals $\Id_1,\ldots,\Id_r$
is the same as the single blowup in the product $\Id_1\cdot
\Id_2\cdots \Id_r$ of ideals. Bodn\'ar \cite{Bo} computes the nonreduced
ideals for a sequence of blowups. First we show that the ideal tangent
cones of a product are sums of ideal tangent cones of the factors of
the product. 

\begin{Lem} \label{L:summe}
Let $\Id=(\vek{x}^{\vek{a}}, \vek{a} \in A )$ and $\Jd=(\vek{x}^{\vek{b}},
\vek{b} \in B )$ with clouds $A,B \subseteq \N^n$ be two ideals in $K[{\bf x}]$.
Then we have $N(\Id \cdot \Jd)= N(\Id)+ N(\Jd)$.
\end{Lem}

\begin{proof}
See \cite[Lemma 2.2]{St}.
\end{proof}

\begin{Lem}   \label{Ltangentenkegel}
Let $\Id_1 = (\vek{x}^\vek{a}: \vek{a} \in A)$ and $\Id_2=(\vek{x}^\vek{b}:
\vek{b} \in B)$ be two monomial ideals in $K[{\bf x}]$ with associated Newton
polyhedra $P=N(\Id_1)$ and $Q=N(\Id_2)$. Then $IT_{\vek{a} + \vek{b}}(\Id_1
\cdot \Id_2) = \IT{\vek{a}}{\Id_1} + IT_{\vek{b}}(\Id_2)$.
\end{Lem}

\begin{proof}
From the definition of the ideal tangent cone it follows that
\[
\begin{split}
IT_{\vek{a}}(\Id_1) & ={}_\N \langle \vek{a}'-\vek{a}: \vek{a}' \in
A \setminus \vek{a} \rangle +\Sigma_n  \,,\\
IT_{\vek{b}}(\Id_2) & =  {}_\N \langle \vek{b}'-\vek{b} :\vek{b}'
\in B \setminus \vek{b} \rangle +\Sigma_n\,.
\end{split}
\]
In $P + Q$ holds
\[
\begin{split}
IT_{\vek{a} + \vek{b}}(\Id) & =  {}_\N \langle \vek{a}'+\vek{b}-(\vek{a} +
\vek{b}), \vek{a} + \vek{b}' - (\vek{a} + \vek{b}), \vek{a}'+\vek{b}' -(\vek{a}
+ \vek{b}) \rangle +\Sigma_n\\
 & = {}_\N \langle \vek{a}'-\vek{a},   \vek{b}'- 
\vek{b} : \vek{a}' \in A \setminus \vek{a}, \vek{b}' \in B
\setminus \vek{b}  \rangle+\Sigma_n\,.
\end{split}
\]
Hence the assertion follows.
\end{proof}

One easily shows that a monomial ideal has a unique minimal set of generators. 
From now on we assume that monomial ideals are generated by a minimal set of
generators.

\begin{defn}
If $\Id ( \vek{x}^\vek{a}: \vek{a} \in A)$ is a monomial ideal with $A$
minimal, then we call $A$ the cloud of $\Id$.
Let $\Id$ and $\Jd$ be two monomial ideals in $K[\vek{x}]$ with clouds $A$ and
$B$. We say that the clouds $A$ and $B$ are transverse if the affine hulls
of $A$ and $B$ are transverse.
\end{defn}

If $A$ and $B$ are two transverse clouds, then we can perform a $\Z$-linear
transformation to achieve that $A \subseteq \N^s \times 0^{n-s}$ and $B
\subseteq 0^s \times \N^{n-s}$ for some integer $s$.

\begin{Lem} \label{einschraenk}
Let $\Id=( \vek{x}^\vek{a}: \vek{a} \in A) \subseteq K[{\bf x}]$ be a monomial
ideal with $A \subseteq \N^s \times 0^{n-s}$. Then the ideal $\Id$ in $K[{\bf
x}]$ is tame if and only if the ideal $\tilde \Id := \Id \cap K[x_1, \ldots,
x_s]$ is tame.
\end{Lem}

\begin{proof}
The cloud of $\Id$ is given by points $\vek{a}= \tilde {\vek{a}} \times 0^{n-s}$
where $\tilde {\vek{a}}$ is an element of the cloud of $\tilde \Id$.
Hence we have $N(\tilde \Id) \times \R^{n-s}=N(\Id)$. The minimal generators of
$\IT{\vek{a}}{\Id}$ for a point $\vek{a} \in N(\Id)$ are thus of the form
$\vek{b}_1, \ldots, \vek{b}_t, \vek{e}_{s+1}, \ldots, \vek{e}_n$. Hence
$\IT{\vek{a}}{\Id}$ and $\IT{\tilde {\vek{a}}}{\tilde \Id}$ are simplicial if
and only if $t=s$.
\end{proof}

\begin{Lem} \label{produkt}
Let $P \subseteq \R^s$ and $Q \subseteq \R^t$ be polyhedra. Let $\vek{e}$ be a vertex of $P$ and $\vek{f}$ be a vertex of $Q$. Then $(\vek{e}, \vek{f} )$ is a vertex of $P \times Q$ in $\R^{s+t}$.
\end{Lem}

\begin{proof}
Because $\vek{e}$ is a vertex of $P$ there exist a vector $\vek{a} \in \R^s$ and a real number $\alpha$ with $\vek{a} \cdot \vek{e} = \alpha$ and $\vek{a} \cdot \vek{p} < \alpha$ for all $\vek{p} \in P \backslash \vek{e}$. Here $\vek{x} \cdot \vek{y}$ denotes the Euclidean scalar product. Similarly, there exists $\vek{b} \in \R^t$ and $\beta \in \R$ such that $\vek{b} \cdot \vek{f} = \vek{\beta}$ and $\vek{b} \cdot \vek{q} < \beta$ for all $\vek{q} \in Q \backslash \vek{f}$. Denoting $\vek{c}= ( \vek{a}, \vek{b}) \in \R^{s+t}$ we have
\begin{equation}
(\vek{e}, \vek{f}) \cdot \vek{c} = \vek{a} \cdot \vek{e} + \vek{b} \cdot \vek{f}= \alpha + \beta\,,
\end{equation}
and for all $(\vek{p}, \vek{q}) \in P \times Q \setminus (\vek{e},\vek{f})$
\begin{equation}
(\vek{p}, \vek{q}) \cdot \vek{c}= \vek{a} \cdot \vek{p} + \vek{b} \cdot \vek{q}
< \alpha + \beta\,.\qedhere
\end{equation}
\end{proof}

\begin{Lem} \label{produktglatt}
Let $\Id_1=(\vek{x}^\vek{a}:\vek{a} \in A)$ and $\Id_2 =
(\vek{x}^\vek{b}: \vek{b} \in B)$ be two monomial ideals in $K[{\bf x}]$
with transverse clouds $A, B \subseteq \N^n$ and associated Newton
polyhedra $P$ resp. $Q$. Then the following hold:
\begin{itemize}
\item[(i)]  If $A \subseteq \N^s \times 0^{n-s}$, $B \subseteq 0^s
  \times \N^{n-s}$ and $\tilde P= P \cap \R^s, \tilde Q= Q \cap
  \R^{n-s}$, then $N(\Id_1 \cdot \Id_2)$ is $\tilde P \times \tilde
  Q$.
\item[(ii)] If $\Id_1$ and $\Id_2$ are tame, then $\Id_1 \cdot \Id_2$
  is tame.
\end{itemize}
\end{Lem}

\begin{proof}
For $(i)$ we can suppose that $A \cup B $ span $\N^n$ (Lemma
\ref{einschraenk}). With Lemma \ref{L:summe} follows
\begin{equation}
N(\Id_1 \cdot \Id_2)= P+Q = \tilde P \oplus \tilde Q  \cong \tilde P
\times \tilde Q.
\end{equation}
Now we show $(ii)$: We may assume the conditions in $(i)$ hold and
thus using Lemma \ref{einschraenk} we have
$\Id_1\subset K[x_1,\ldots,x_s]$ and $\Id_2\subset
K[x_{s+1},\ldots,x_{n}]$. If $\tilde {\vek{a}}$ is a vertex of $\tilde P$
resp. $\tilde {\vek{b}}$ is a vertex of $\tilde Q$ then Lemma
\ref{produkt} yields that $(\tilde{\vek{a}},\tilde{\vek{b}})$ is a vertex
of $N(\Id_1
\cdot \Id_2 )$. Conversely each vertex of $N(\Id_1 \cdot \Id_2)$ is a
sum of two vertices $\tilde{\vek{a}} \times 0^s \in P$ and $0^{n-s}
\times \tilde{\vek{b}} \in Q$. Hence the set of vertices of $N(\Id_1
\cdot \Id_2)$ is $\{ \vek{a} + \vek{b}: \vek{a}$ a vertex of $P$,
$\vek{b}$ a vertex of $Q \}$. Using Lemma \ref{Ltangentenkegel} and
the transversality we have
$\IT{\vek{a}+\vek{b}}{\Id_1 \cdot \Id_2} =IT_{{\bf a}}(\Id_1)\oplus IT_{{\bf
    b}}(\Id_2)$, so that the minimal generators are of the form
$({\bf v},0)$ with ${\bf v}$ a minimal generator for $IT_{{\bf
    a}}(\Id_1)$ and $(0,{\bf w})$ with ${\bf w}$ a minimal generator
for $IT_{{\bf b}}(\Id_2)$. Thus there are precisely $s+(n-s)=n$
minimal generators.
\end{proof}

\begin{Snum}  \label{Spartition}
Let $\Id_i$ for $i=1, \ldots, k$ be monomial ideals in $K[{\bf x}]$ with
pairwise transverse clouds. Then $\Id:= \prod_{i=1}^k \Id_i$ is tame
if each $\Id_i$ is tame.
\end{Snum}

\begin{proof}
Let $\Id_i=( \vek{x}^\vek{a}:
\vek{a} \in A_i)$ be ideals in $K[{\bf x}]$ with pairwise transverse
clouds. We use induction over $k$: For $k=1$ the assertion is
trivial. For the induction step $k$ to $k+1$ let the first $k$ clouds
be contained in $\N^s \times 0^{n-s}$ for an $s \leq n$. This implies
$A_{k+1} \subseteq 0^s \times \N^{n-s}$. By the induction assumption,
the product ideal $\prod_{i=1}^k \Id_i$ is tame. Because the cloud of
$\Id$ is $A= A_{k+1}\cup\bigcup_{i=1}^k A_i $ and
$ \bigcup_{i=1}^k A_i $ and $A_{k+1}$ are transverse one
can use Lemma \ref{produktglatt} to conclude $\Id$ is tame.
\end{proof}

\begin{ex}
Let $\Id_1=(x,y)$ and $\Id_2=(z^2,zw,w^3)$ be ideals in
$K[x,y,z,w]$; $\Id_1$ is the reduced ideal
of the $zw$-plane in $\A^4$ and the nonreduced ideal $\Id_2$
defines the $xy$-plane in $\A^4$. One easily computes that both ideals
are tame. Using Proposition \ref{Spartition}, one finds that the
blowup of $\A^4$ with center
\begin{equation}
\Id_1 \cdot \Id_2 =(xz^2,yz^2,xzw,yzw,xw^3,yw^3)\,,
\end{equation}
is smooth. The zero set of $\Id_1 \cdot \Id_2$ is the union of the
$zw$- and the $xy$-plane.
\end{ex}

\begin{ex}
Let $\Id_1=(x,y)$ and $\Id_2=(x,z)$ be the reduced ideals of
the $z$- resp. $y$-axis in affine three-dimensional space with
coordinates $x,y,z$. Both ideals are tame but the clouds are not
transverse and the blowup $\widetilde
\A^3$ in the product $\Id_1 \cdot \Id_2=(x^2,xy,xz,yz)$, a nonreduced
structure on the coordinate cross, is not smooth. To see this,
consider the $yz$-chart: The ideal tangent cone $\IT{\vek{e}_2 +
  \vek{e}_3}{\Id_1 \cdot \Id_2}$ has the minimal system of
generators $\vek{e}_1 - \vek{e}_2$, $\vek{e}_1 - \vek{e}_3$,
$\vek{e}_2, \vek{e}_3$. Hence $\IT{\vek{e}_2 + \vek{e}_3}{\Id_1 \cdot
  \Id_2}$ is not simplicial. In the $yz$-chart $\widetilde \A^3= \Spec
( K[x,y,z,w] / (xz -yw))$ has an isolated (simple toric) singularity.
\end{ex}

Let $\Id = ({\bf x}^{\bf a} : {\bf a}\in A)$ be a monomial ideal in
$K[x_1,\ldots,x_n]$. A direct application of Remark \ref{Rmk:summe}
implies that the vertices of $N(\Id^2)$ are of the form $2{\bf a}$ for
some vertex ${\bf a}$ of $N(\Id)$. Since $IT_{2{\bf a}} (\Id^2) =
IT_{{\bf a}}(\Id)$, either both are pointed or both are not
pointed, which shows that ${\bf a}$ is a vertex of $N(\Id)$ if and only if
$2{\bf a}$ is vertex of $N(\Id^2)$. Furthermore, $IT_{2{\bf a}}
(\Id^2)$ is simplicial if and only if $ IT_{{\bf a}}(\Id)$ is
simplicial and thus $\Id$ is tame if and only if $\Id^2$ is tame. This
can also be proven by more general methods for homogeneous ideals, as
for example \cite[ex. 5.10 (b),II]{Hs}, \cite[ex. III-15]{EH}. We have
an easy generalization:

\begin{Lem}\label{Lem.powers}
Let $\Id_1,\ldots,\Id_k$ be monomial ideals in $K[x_1,\ldots,x_n]$ and
let $\alpha_1,\ldots,\alpha_k$ be some positive integers. Then
$\Id=\Id_1\cdot \ldots \cdot \Id_k$ is smooth if and only if
$\Jd=\Id_{1}^{\alpha_1}\cdot \ldots \cdot\Id_{k}^{\alpha_k}$ is
smooth.
\end{Lem}

\begin{proof}
Any vertex of $\Jd$ is of the form ${\bf p}=\sum \alpha_i {\bf a}_{i}$ for
${\bf a}_i$ a vertex of $N(\Id_i)$. The ideal tangent cone of ${\bf q}=\sum
{\bf a}_{i}$ in the cloud of $\Id$ is the same as the ideal tangent
cone of ${\bf p}$ in the cloud of $\Jd$. In
particular, if one is pointed, then so is the other, and thus ${\bf
  p}$ is a vertex precisely when ${\bf q}$ is a vertex. Furthermore,
if $IT_{{\bf p}}(\Jd)$ is simplicial, then so is $IT_{{\bf q}}(\Id)$.
\end{proof}


\subsection{Coordinate ideals}


A coordinate ideal in $K[{\bf x}]$ is an ideal generated by a subset of
the $\xin$.  In this section we study products of
coordinate ideals. Let $\Id$ be a coordinate ideal in $K[{\bf x}]$. In
the Newton polyhedron $N(\Id)=\conv( \{ \vek{e}_i: i \in I \} +
\R^n_+)$ the ideal tangent cone of a vertex is of the form
\begin{equation}
\IT{\vek{e}_i}{\Id}= {}_\N\langle \vek{e}_j - \vek{e}_i: j \in I
\setminus i \rangle + \Sigma_n\,.\label{ITkoord} 
\end{equation}
There are $n$ minimal generators among these, and thus
$\IT{\vek{e}_i}{\Id}$ is simplicial.

For a coordinate ideal $\Id = (x_i : i\in
I) \subset K[x_1,\ldots,x_n]$ we identify the cloud $A = \left\{ {\bf
    e}_i :i\in I\right\}$ with the subset $I$ of
$\left\{1,2,\ldots,n\right\}$ by identifying $i\in I$ with ${\bf
  e}_i\in A$. For this reason we call $I$ the cloud of
$\Id$. Two coordinate ideals $\Id = (x_i : i\in I)$ and $\Jd = (x_j :
j\in J)$ have transverse clouds if and only if $I\cap
J=\emptyset$. By Lemma \ref{produktglatt} the product of coordinate
ideals with transverse clouds is tame.

Let $\Id= \prod_{i=1}^s \Id_i$ be a product of coordinate ideals
$\Id_i=(x_j : j \in I_i)$ with $I_i \subseteq \{ 1 , \ldots, n
\}$. Then each monomial generator of $\Id$ is a product of $s$
monomials $x_j$. All generators of $\Id$ thus look like $x_{j_1}
\cdots x_{j_s}$, with $j_i \in I_i$. From this we see that each
vertex $\vek{v}$ of the Newton polyhedron is a sum of $s$
standard basis vectors. Hence all vertices lie on the affine
hyperplane $H_s:=\{ \vek{v} \in \R^n: \vek{v} \cdot \vek{1}=s
\}$. Here $\vek{1}$ denotes the vector $(1, \ldots, 1)$.

\begin{Lem} \label{Lunimodular}
Let $\Id$ be a product of $s$ coordinate ideals $\Id_i=(x_j : j \in
I_i)$ in $K[{\bf x}]$. Each minimal generator of the ideal tangent cone
of the Newton polyhedron $N(\Id)$ in a point $\vek{a}= \sum_k \alpha_k
\vek{e}_k$ with $\sum_k \alpha_k=s$ and $\vek{a}$ an element of the
cloud of $\Id$, is of the form $\vek{e}_\beta$ resp. $\vek{e}_\beta -
\vek{e}_k$ for a $k$ with $\alpha_k \neq 0$ and some $\beta \in \{ 1,
\ldots, n \}$. At least one standard basis vector $\vek{e}_j$ is a
minimal generator of $\IT{\vek{a}}{\Id}$.
\end{Lem}

\begin{proof}
The form of the minimal generators follows from equation (\ref{ITkoord})
with Lemma \ref{Ltangentenkegel}. Since $_\R\! \langle \vek{e}_i -
\vek{e}_j: i \neq j \rangle$ is $(n-1)$-dimensional and the minimal
generators form a basis of $\R^n$, not all minimal generators can be of
the form $ \vek{e}_i - \vek{e}_j $.
\end{proof}

\begin{defn}
For ${\bf v}=(v_1,\ldots,v_n)\in \Z^n$ we define the norm
$|{\bf v}| = \sum_{i=1}^{n}v_i$.
\end{defn}

Note that on ideal tangent cones of
coordinate ideals and products of them $|.|$ is nonnegative. The
minimal generators are of norm $1$ and $0$. We say a minimal generator
points in the direction $i$ if it
is ${\bf e}_i$ or of the form ${\bf e}_i - {\bf e}_j$ for some $j$.
Since all standard basis vectors are in the ideal tangent cone we need
at least a minimal generator in the direction $i$ for all $i$. By the
pigeon hole principle, the ideal tangent cone is not simplicial if for
some $i$ there are at least two minimal generators pointing in
direction $i$. The only way this can happen is that ${\bf e}_i - {\bf
  e}_k$ and ${\bf e}_i - {\bf e}_l$ are minimal generators for two
different $k$ and $l$. Hence, only norm zero minimal generators obstruct
simplicity for products of coordinate ideals.

\begin{Snum}  \label{Senthalten}
Let $\Id = \prod_{i=1}^{s} \Id_i$ be a product of coordinate ideals
$\Id_i = (x_k : k\in I_i)$ and let $\Jd$ be a coordinate ideal
containing $\Id$. If $\Id$ is tame, then so is $\Id\cdot \Jd$.
\end{Snum}

\begin{proof}
By Lemma \ref{einschraenk} we may assume $\Jd$ contains all $x_m$ not in
$I = \cup_i I_i$.
Any vertex of $N(\Id)$ is of the form ${\bf p}=\sum_{i=1}^{s}{\bf
  e}_{k_i}$ with $k_i \in I_i$. By Lemma \ref{Lunimodular} the
minimal generators of $IT_{{\bf p}}(\Id)$ are either of the form ${\bf
e}_{l_i} - {\bf e}_{k_i}$ for some $l_i\in I_i$ or of the form $e_l$
and by assumption there are precisely $n$ of
them. We claim that the minimal generators of $IT_{{\bf p}+{\bf
e_j}}(\Id\cdot \Jd)$ can be obtained from the minimal generators
of $IT_{{\bf p}}(\Id)$ as follows: all minimal generators of the form
${\bf e}_l$ for $l\neq j$ are replaced by ${\bf e}_l - {\bf e}_j$ and
all others are left unchanged. Since by Lemma \ref{Ltangentenkegel} we have
\begin{equation}
IT_{{\bf e}_j}(\Jd) = {}_{\N}\langle {{\bf e}_j,\bf e}_{m\setminus j}
- {\bf e}_j: m\neq j \rangle\,,\quad{\rm and}\quad
IT_{{\bf p}+{\bf e_j}}(\Id\cdot \Jd)= IT_{{\bf p}}(\Id) + IT_{{\bf e}_j}(\Jd)
\,,
\end{equation}
all new generators are in $IT_{{\bf p}+{\bf e_j}}(\Id\cdot \Jd)$.
There are still $n$ minimal generators and thus it suffices to prove
the claim.

We first prove that ${\bf e}_j$ is a minimal generator of  $IT_{{\bf
    p}}(\Id)$ and hence a new minimal generator. If $j\notin I$, then
in the expansion of ${\bf e}_j$ in minimal generators no ${\bf
  e}_{l_i} - {\bf e}_{k_i}$ can occur and hence ${\bf e}_j$ is a
minimal generator. If $j\in I$, say $j\in I_1$, then we must have
$j=k_1$. Indeed, if not, then by Remark \ref{Rmk:summe} we see that ${\bf
  e}_{k_1}+{\bf e}_j$ is not a vertex of $N(\Id_1\cdot \Jd)$, but then
${\bf p}+{\bf e}_j$ is not a vertex of $N(\Id\cdot \Jd)$. If ${\bf
  e}_{k_1}$ is not minimal in $IT_{{\bf
    p}}(\Id)$, then ${\bf e}_{k_1} - {\bf e}_{k_i}$ must
be minimal for some $i$. But then $j=k_1\in I_i$ and thus $j=k_i$ so
that ${\bf e}_{k_1} - {\bf e}_{k_i}=0$. Hence ${\bf e}_j = {\bf e}_{k_1}$
is minimal.

As ${\bf e}_j$ is a new minimal generator, $IT_{{\bf p}}(\Id)$ is in
the span of the new minimal generators. Hence we only have to show
that all vectors ${\bf e}_a - {\bf e}_j$ are generated by the new
minimal generators. The expansion of ${\bf e}_a$ in the minimal
generators of $IT_{{\bf p}}(\Id)$ contains precisely one minimal
generator of the form ${\bf e}_l$ since else the norm
would not equal $1$. If $l=j$
we delete it from the expansion and if $l\neq j$ we substitute ${\bf
  e}_l$ with ${\bf e}_l -{\bf e}_j$. In both cases we are done.
\end{proof}

\begin{Kor} \label{Lenthalten}
Consider two coordinate ideals $\Id=( x_i: i \in I)$ and $\Jd=(x_j:
j \in J)$ in $K[{\bf x}]$ with $I \subseteq J$. Then $\Id \cdot \Jd$ is
tame.
\end{Kor}

\begin{Snum} \label{Sprod2}
Let $\Id$ and $\Jd$ be two coordinate ideals in $K[x_1,\ldots,x_n]$
with clouds $I$ and $J$, respectively. Then the product $\Id\cdot \Jd$
is tame if and only if either $I\cap J=\emptyset$ or one of the clouds
is contained in the other. 
\end{Snum}

\begin{proof}
We may assume $I\cup J = \left\{ 1,\ldots,n\right\}$. By
Lemma \ref{produktglatt} and Corollary \ref{Lenthalten} the {\it
  if}-part of the
proof is done. For the {\it only if}-part we remark that for two
subsets $I,J$ of $\left\{ 1,\ldots,n \right\}$ there are three
possibilities: either $(i)$ they are disjoint, or $(ii)$ one is contained
in the other, or $(iii)$ $I\setminus J$, $J\setminus I$ and $I\cap J$ are not
empty. We will show that in the third case there is a nonsimplicial
vertex in the Newton polyhedron.

We may assume $1\in I\setminus J$, $2\in J\setminus I$. By Lemmas
\ref{Ltangentenkegel} and \ref{Lunimodular} the ideal
tangent cone of
${\bf e}_1+{\bf e}_2$, which is a vertex, is given by
\begin{equation}
IT_{{\bf e}_1+{\bf e}_2}(\Id\cdot \Jd) = {}_{\N}\langle {\bf e}_r-{\bf e}_1,{\bf
  e}_s - {\bf e}_2,{\bf e}_1,{\bf e}_2: r\in I\setminus
\left\{1\right\},s\in J\setminus \left\{2\right\} \rangle \,.
\end{equation}
Choose $c\in I\cup J$, then we claim that ${\bf v}_1={\bf e}_c - {\bf
  e}_1$ and ${\bf v}_2={\bf e}_c - {\bf e}_2$ are both minimal and
thus $IT_{{\bf e}_1+{\bf e}_2}(\Id\cdot \Jd)$ is not
simplicial. If ${\bf v}_1$ would not be minimal we can subtract some
generator and the result is still in $IT_{{\bf e}_1+{\bf e}_2}(\Id\cdot
\Jd)$. Since $|{\bf v}_1|=0$ we cannot subtract ${\bf e}_1$ or ${\bf
  e}_2$. Since only the first and the second component of elements of
$IT_{{\bf e}_1+{\bf e}_2}(\Id\cdot \Jd)$ can be negative we can only subtract
${\bf v}_2$ to obtain ${\bf e}_2 - {\bf e}_1$. But $|{\bf e}_2 - {\bf
  e}_1|=0$ and the only generator with positive second entry is ${\bf
  e}_2$; thus ${\bf e}_2 - {\bf e}_1$ is not in $IT_{{\bf
    e}_1+{\bf e}_2}(\Id\cdot \Jd)$. Thus ${\bf v}_1$ is minimal and a similar
argument shows ${\bf v}_2$ is minimal.
\end{proof}

For three coordinate ideals we can determine explicitly
the conditions that have to be satisfied in order that the product is
tame. We already know that when all three have disjoint clouds, that
the product is smooth. Also, in the case where one cloud is contained
in another and the third cloud is disjoint from these, then the
product is tame. The remaining possibilities are dealt with below.

\begin{Lem}\label{Dreier}
Let $\Id$, $\Jd$ and $\ms{K}$ be three monomial coordinate ideals in
$K[x_1,\ldots,x_n]$ with clouds $I,J$ and $K$ respectively. Suppose
that $I\cup J = K$, then the product $\Id\cdot \Jd \cdot \ms{K}$ is
tame.
\end{Lem}

\begin{proof}
If $I\subset J$, then $J=K$ and the product is tame by Corollary
\ref{Lenthalten}. Hence we assume $I\nsubseteq J$ and $J\nsubseteq I$. Let
${\bf v}$ be a vertex of $N=N(\Id\cdot \Jd \cdot \ms{K})$. Then ${\bf
  v} = \nu(I)+\nu(J)+\nu(K)$, where $\nu(I)$, $\nu(J)$ and $\nu(K)$
are vertices of the Newton polyhedra of $\Id$, $\Jd$ and $\ms{K}$,
respectively. By Remark \ref{Rmk:summe} we need that $\nu(I)=\nu(K)$ or
$\nu(J) = \nu(K)$; if $\nu(K)\in I$, then $\nu(K)+\nu(I) =
\tfrac{1}{2}(2\nu(K)) + \tfrac{1}{2}(2\nu(I))$. Hence we may assume
$\nu(K) = \nu(I)$. But then it follows that $\nu(J) - \nu(I)$ is in
$IT_{{\bf v}}(\Id\cdot \Jd \cdot \ms{K})$. If $IT_{{\bf v}}(\Id\cdot
\Jd \cdot \ms{K})$ is not simplicial, there is an index $1\leq i\leq
n$ such that there are two minimal generators ${\bf e}_i - {\bf
  e}_{k}$ and ${\bf e}_{i} - {\bf e}_{l}$ for two different indices
$1\leq k , l \leq n$. But the only possible choices $k$ and $l$ are
$\nu(I)$ and $\nu(J)$, in case of which ${\bf e}_i - \nu(I) = {\bf
  e}_i - \nu(J) + \nu(J) - \nu(I)$. Hence no $i$ can exist for which
${\bf e}_i - \nu(I)$ and ${\bf e}_i - \nu(J)$ are minimal generators
and $IT_{{\bf v}}(\Id\cdot \Jd \cdot \ms{K})$ is simplicial.
\end{proof}

\begin{Lem}
Let $\Id,\Jd$ and $\ms{K}$ be three monomial coordinate ideals
in $K[x_1,\ldots,x_n]$ with clouds $I,J$ and $K$, respectively. When
the inclusions $I\subset J\cup K$, $J\subset K\cup I$ and $K\subset
I\cup J$ hold, then the product $\Id\cdot \Jd \cdot \ms{K}$ is tame.
\end{Lem}

\begin{proof}
If $I\subset J$, then it follows that $K\subset J$ and thus $J=K\cup I$, in
 which case the product is tame by Lemma \ref{Dreier}. Hence for the remainder
we assume that no cloud is contained in another cloud.

Let ${\bf v}$ be a vertex of the Newton polyhedron of $\Id\cdot \Jd
\cdot \ms{K}$. Then ${\bf v} = \nu(I)+\nu(J)+\nu(K)$, where $\nu(I)$,
$\nu(J)$ and $\nu(K)$ denote vertices of the Newton polyhedron of
$\Id$, $\Jd$ and $\ms{K}$, respectively. The ideal tangent cone of
$\Id\cdot \Jd \cdot \ms{K}$ at ${\bf v}$ is the sum
\begin{equation}
IT_{{\bf v}}(\Id\cdot \Jd \cdot \ms{K}) = IT_{\nu(I)}(\Id) +
IT_{\nu(J)}(\Jd) + IT_{\nu(K)}(\ms{K})\,.
\end{equation}
First assume that $\nu(I)$, $\nu(J)$ and $\nu(K)$ are distinct. We may
assume $\nu(I)$ is contained in $J$. Then $\nu(I)-\nu(J)\in
IT_{\nu(J)}(\Jd)$ and thus $\nu(I)-\nu(J) \in IT_{{\bf v}}(\Id\cdot
\Jd \cdot \ms{K})$. If $\nu(J)\in I$ then also $\nu(J) - \nu(I)\in
IT_{{\bf v}}(\Id\cdot \Jd \cdot \ms{K})$ and ${\bf v}$ is not a
vertex. Hence we need $\nu(J)\subset K\setminus I$ and by the same
reasoning $\nu(K)\in I\setminus J$. But then the vectors $\nu(J) -
\nu(K)$ and $\nu(K) - \nu(I)$ and thus their sum $\nu(J) - \nu(I)$ are
contained in $IT_{{\bf v}}(\Id\cdot \Jd \cdot \ms{K})$ and ${\bf v}$
is not a vertex. Hence $\nu(I)$, $\nu(J)$ and $\nu(K)$ cannot be
distinct.

We may assume $\nu(I) = \nu(J)$ and $\nu(K)\in I$, so that $\nu(K) -
\nu(I)$ is in $IT_{{\bf v}}(\Id\cdot \Jd \cdot \ms{K})$. If $IT_{{\bf
v}}(\Id\cdot \Jd \cdot \ms{K})$ is not simplicial then there must
be two minimal generators of the form ${\bf w} - \nu(I)$ and ${\bf w}
-\nu(K)$ for some ${\bf w}\in \mathbb{Z}^n$, but ${\bf w} - \nu(I)=
{\bf w} - \nu(K ) + \nu(K) - \nu(I)$. Hence $IT_{{\bf v}}(\Id\cdot \Jd
\cdot \ms{K})$ is simplicial.
\end{proof}

Explicit examples show that the above treated criteria are
exhaustive: for any three subsets $I$, $J$, and $K$ of
$\{1,\ldots,n\}$ not satisfying any of the criteria
\begin{itemize}
\item[(i)] $I$, $J$ and $K$ are disjoint,
\item[(ii)] $I\subset J$ and $J\cap K=\emptyset$,
\item[(iii)] $I\subset J\cup K$, $I\subset J\cup K$ and $I\subset
  J\cup K$,
\item[(iv)] $K = I\cup J$,
\end{itemize}
one can find examples of three ideals such that the
product is not tame.

We close this section with a result that will play a role when we
discuss permutohedral blowups and provides a smoothing procedure in Proposition
\ref{propSmooth}.

\begin{Lem} \label{LpaarweiseVereinigung}
Let $\Id_1,\ldots,\Id_N$ be coordinate ideals in
$K[x_1,\ldots,x_n]$. Then the ideal $\Id=\prod_{i<j}(\Id_i+\Id_j)$ is
smooth.
\end{Lem}

\begin{proof}
We write $I_i$ for the cloud of $\Id_i$, $P_i$ for the Newton
polyhedron of $\Id_i$ and $P$ for the Newton polyhedron of $\Id$.
We introduce the set $\Omega$ of unordered pairs $(i,j)$, where $i$
and $j$ are distinct integers running from $1$ to $N$.

We fix a vertex ${\bf a}$ of
$P$; ${\bf a}$ is of the
form $\sum_{i<j}{\bf e}_{k_{ij}}$, where $k_{ij}$ is in the cloud of
$\Id_i+\Id_j$, and hence in the cloud of $I_i$ of in the cloud of
$I_j$ (or in both). The ideal tangent cone $IT_{{\bf a}}(\Id)$ is the
sum of the cones $IT_{{\bf e}_{k_{ij}}}(\Id_i+\Id_j)$, which is
generated by all vectors of the form ${\bf e}_m - {\bf e}_{k_{ij}}$,
where $m$ runs over all elements in $I_i\cup I_j$, and all standard
basis vectors.

We construct a function
$w:\Omega\to\left\{1,\ldots,N\right\}$ as follows: for $(i,j)$ we put
$w((i,j))=i$ if $k_{ij}$ is in $I_i$ and $w((i,j))=j$ if $k_{ij}\in
I_j$ and when $k_{ij}\in I_i\cap I_j$ we make a choice - for example
$w((i,j))={\rm min}(i,j)$. The function $w$ describes from which
Newton polyhedron we have a contribution to ${\bf a}$. We call $W$ the
graph of $w$ and we write $w_{ij}=w((i,j))$. Next we define a function
$\nu: W\to \left\{{\bf e}_1,\ldots,{\bf e}_n\right\}$ as follows: we assign
$(i,j,w_{ij})\in W$ the basis vector ${\bf e}_{k_{ij}}$ from the expansion
of ${\bf a}$. We now claim that when there are different $k$
and $l$ such that $(1,k,1)$ and $(1,l,1)$ are in $W$, then
$\nu(1,k,1)=\nu(1,l,1)$. Suppose ${\bf e}_r=\nu(1,k,1)$ and ${\bf
  e}_s=\nu(1,l,1)$ are different, then ${\bf e}_r+{\bf e}_s =
\tfrac{1}{2}(2{\bf e}_r) + \tfrac{1}{2}(2{\bf e}_s)$ is not a vertex of
$(P_1\cup P_k)+(P_1\cup P_l)$, which is the Newton polyhedron of
$(\Id_1+\Id_k)(\Id_1+\Id_l)$. But then ${\bf a}={\bf e}_r+{\bf e}_s +
\ldots$ cannot be a vertex of $\Id$. We can write for ${\bf a}$
\begin{equation}
{\bf a}= \sum_{k\neq 1}\nu(1,k,1) + \sum_{k\neq 2} \nu(2,k,2) +
\ldots+\sum_{k\neq N}\nu(N,k,N)\,.
\end{equation}
When $w(i,j)=i$ and $\nu(i,j,i)={\bf e}_r$, then the ideal tangent
cone $IT_{{\bf a}}(\Id)$ contains generators ${\bf e}_m - {\bf e}_r$,
where $m$ runs over all points in $I_i\cup I_j$. Conversely, when
${\bf e}_m - {\bf e}_r$ is a minimal generator in $IT_{{\bf a}}(\Id)$,
then ${{\bf e}}_r=\nu(i,j,w_{ij})$ for some $i,j$.

Assume that $IT_{{\bf a}}(\Id)$ is not smooth. Then there is a
direction $i$ for which there are too many
minimal generators. We may assume that ${\bf e}_i - {\bf e}_1$ and
${\bf e}_i - {\bf e}_2$ are two minimal generators. Thus $1$ and $2$
are in the image of $w$. Then there are $k,k',l,l'$ such that ${\bf
  e}_1=\nu(k,l,k)$ and ${\bf e}_2=\nu(k',l',k')$ and $k\neq
k'$. We may assume $w(k,k')=k$ so that $\nu(k,k',k)={\bf e}_1$. Both
${\bf e}_1$ and ${\bf e}_2$ are in the Newton polyhedron of $\Id_k +
\Id_{k'}$ and ${\bf e}_{2} - {\bf e}_1$ is in the ideal tangent cone
of ${\bf a}$, which contradicts that both ${\bf e}_i - {\bf e}_1$ and
${\bf e}_i - {\bf e}_2$ are minimal generators.
\end{proof}

\begin{Kor} \label{Lganzesbs}
Let $\Id$ be the monomial ideal from Lemma \ref{LpaarweiseVereinigung}
and $\Jd = \prod_{i=1}^{k} \Id_i$. Then $\Id \cdot \Jd$ is tame.
\end{Kor}

\begin{proof}
Take the ideal $\Id_{k+1}=(0)$ as the $(k+1)$th factor in the product. With
Lemma \ref{LpaarweiseVereinigung} it then follows that $\Id \cdot
\Jd=\prod_{1\leq i<j\leq k+1} (\Id_{i}+\Id_j)$ is tame.
\end{proof}

\begin{ex}\label{Bsp.planes.perm}
Consider the monomial ideals $\Id_i = (x_i)$ in $K[x_1,\ldots,x_n]$. By Lemma
\ref{LpaarweiseVereinigung} the ideal $\prod_{i<j}(x_i , x_j)$ is
tame. The ideal $\Id$ is symmetric under the action of the symmetric
group on $\{1,\ldots,n\}$. In Section \ref{SECpermutation} we will
discuss a wider class of such symmetric ideals and prove that they are tame.
\end{ex}

Given any set of coordinate ideals $\Id_1,\ldots , \Id_k$ in $K[{\bf
  x}]$, one easily sees that $V(\Id_i + \Id_j) \subset
V(\prod_{l=1}^{k}\Id_l)$. It follows that we have
\begin{equation}
V\Bigl(\prod_{i=1}^{k}\Id_i \cdot \prod_{1\leq i<j\leq k}(\Id_i +
\Id_j)\Bigr)=V\Bigl(\prod_{i=1}^{k}\Id_i\Bigr)\,. 
\end{equation}
Hence we arrive at the following smoothing procedure:

\begin{Snum}\label{propSmooth}
Suppose that the product $\prod_{i=1}^{k}\Id_i$ of coordinate ideals
$\Id_1,\ldots, \Id_k$ in $K[{\bf x}]$ is not tame, then the ideal 
\[
\prod_{i=1}^{k}\Id_i \cdot \prod_{1\leq i<j\leq k}(\Id_i + \Id_j)
\]
has the same zero set as $\prod_{i=1}^{k}\Id_i$ and is tame.
\end{Snum}


\section{Blowups in monomial Building Sets} \label{SECbuilding}


Fulton and MacPherson study the problem of the compactification of the
complement of configuration spaces \cite{FM}. They consider
configuration spaces $F(n,X)$ of smooth algebraic varieties $X$, where
$F(n,X)$ denotes  the space of $n$-tuples of mutually distinct points
in $X$: 
$$ F(n,X)= \{ (\xin) \in X^n: x_i \neq x_j \text{ for } i \neq j \}.$$
One can construct a compactification $X[n]$ of $F(n,X)$ via a sequence of
blowups such that the complement of the original configuration space is a normal
crossings divisor. De Concini and Procesi \cite{DP} have introduced wonderful
models of finite families of linear subspaces of a vector space $X$. A
wonderful model for
such a  subspace arrangement is constructed so that the complement of
the arrangement remains unchanged and the subspaces are substituted by a normal
crossings divisor. In order to achieve normal crossings, De Concini and Procesi
introduced the notion of a building set of an arrangement of linear subspaces.

Later MacPherson and Procesi \cite{MP} generalized wonderful models to conic
varieties over $\C$. They use conical stratifications and generalize the notion
of building sets. Special compactifications have also been studied by Hu
\cite{Hu} and Ulyanov \cite{Ul}. Li \cite{LiLi} has generalized wonderful
models to compactifications of subspace arrangements of a smooth
variety. In \cite[Thm 1.3]{LiLi} it is shown that the blowup in a product of
ideals is smooth if they form a building set.
Wonderful compactifications are also studied by combinatorial means by Feichtner
\cite{Fe}. 

In the following we will introduce the notion of building sets, like in
\cite{DP}. We use this notion for linear subspaces when we look at
zero sets of coordinate ideals $\Id=(x_i, i \in I)$ in $K[{\bf x}]$ with $I
\subseteq \{1, \ldots, n \}$, which yield only linear (coordinate) subspaces.
With the help of building sets we find a large class of tame product ideals:  In Proposition
\ref{S:buildingsmooth} we show that the blowup
with center a product $\prod_i \ms{G}_i$ of ideals $\ms{G}_i=(x_j: j \in I_i)$
is smooth if $\mc{G}=\{ \ms{G}_i \}$ form a building set. However, in Section
\ref{SECpermutation} we provide a large class of a tame products of coordinate
ideals, which do not come from building sets.

Consider a vector space $V \cong K^n$. An {\it arrangement} of subspaces is a
finite set $\mc{C}$ of nonzero subspaces of $V$ closed under taking sums. 

For $U \in \mc{C}$ we call $U_1, \ldots, U_k \in \mc{C}$, with $k \geq 2$, a  decomposition of $U$ if
$$U= U_1 \oplus \cdots \oplus U_k\,,$$
and for each subspace $A \subseteq U$ in $\mc{C}$ also each $A \cap U_i$ is
an element of $\mc{C}$ and
$$A=(A \cap U_1) \oplus \cdots \oplus (A \cap U_k).$$
If a subspace $U$ does not allow a decomposition in $\mc{C}$ we call $U$ 
irreducible. Given nonzero subspaces $V_1,\ldots,V_k$, then the arrangement
generated by $V_1,\ldots,V_k$ is the smallest set of subspaces containing
$V_1,\ldots,V_k$ and is closed under taking sums.

\begin{ex}
(1) Let $V=K^4$ with the standard basis  $\vek{e}_1, \ldots, \vek{e}_4.$ Let
$\mc{C}=\{ U_1, \ldots, U_7 \}$ be an arrangement with $U_1=\,_K \langle
\vek{e}_1, \vek{e}_2, \vek{e}_3 \rangle$, $U_2=\,_K \langle \vek{e}_1\rangle$,
$U_3=\,_K \langle \vek{e}_2 \rangle$, $U_4=\,_K \langle \vek{e}_3 \rangle$,
$U_5=\,_K \langle \vek{e}_1, \vek{e}_2 \rangle$, $U_6=\,_K\langle \vek{e}_2,
\vek{e}_3 \rangle$ and $U_7=\,_K\langle \vek{e}_1, \vek{e}_3 \rangle$. The
irreducible elements of $\mc{C}$ are $U_2, U_3, U_4$. Then $U_1=U_5\oplus U_4$
and $U_1=U_2 \oplus U_3 \oplus U_4$ are decompositions of $U_1$, since $U_6=(U_6
\cap U_4) \oplus (U_6 \cap U_5)$ and $U_7=(U_7 \cap U_4) \oplus (U_7 \cap U_5)$
are direct sums of elements in $\mc{C}$.  \\
(2) Now consider $\mc{C}'=\mc{C} \backslash \{ U_{3} \}$. The direct sum $U_1=U_4\oplus U_5$ is not a decomposition since $U_{6} \subseteq U_1$ in $\mc{C}'$, but $U_6 \cap U_5=\,_K \langle \vek{e}_2 \rangle \not \in \mc{C}'$.
\end{ex}

\begin{defn} \label{Sbsdefinition}
Let $\mc{G}$ be a nonempty set of nonzero subspaces of $V$. Let $\mc{C}$ be the
arrangement generated by $\mc{G}$. We call $\mc{G}$ a building set if the
following conditions are satisfied: 
\begin{itemize}
\item[(a)] The irreducible elements of $\mc{C}$ are contained in $\mc{G}$.
\item[(b)] If $A,B$ are in $\mc{G}$ and if $A+B$ is not a direct sum, then $A+B$
is already in $\mc{G}$.
\end{itemize}

Equivalently, $\mc{G}$ is a building set of the arrangement $\mc{C}$ if and only
if each element $C \in \mc{C}$ is the direct sum $C=\bigoplus_{i=1}^k G_i$ of
the set of maximal elements $G_1, \ldots, G_k$ in $\mc{G}$ that are contained in
$C$ (see \cite[thm 2.3]{DP}).
\end{defn}

Note that the decomposition of $C=\bigoplus_i G_i$ with the $G_i$
maximal with respect to inclusion is unique. Moreover, for all $A$ in $\mc{G}$, $A \subseteq C$ there exists an $i$ such that $A \subseteq G_i$.

\begin{ex}
Let $\mc{C}$ be an arrangement of subspaces in $V$. \\
(1) $\mc{C}$ is a building set. This is clear because one can write each $C \in \mc{C}$ as a ``direct sum'' of its maximal subspace $C$. \\
(2) The set of irreducible elements of $\mc{C}$ is a building set. \\
(3)  Let $V=K^3$ with basis $\vek{e}_1, \vek{e}_2, \vek{e}_3$ and let $\mc{G}=\{G_1,G_2,$ $G_3\}=$ $\{\,_K\langle \vek{e}_1, \vek{e}_2 \rangle,$ $_K\langle \vek{e}_1, \vek{e}_3 \rangle,$ $_K\langle \vek{e}_2, \vek{e}_3 \rangle\}$. Then $\mc{G}$ is not a building set, for $$K^3=\,_K\langle \vek{e}_1, \vek{e}_2, \vek{e}_3 \rangle=G_1+G_2=G_1+G_3=G_2+G_3$$
is not a \textit{direct} sum of maximal elements of $\mc{G}$. \\
However, if we consider $\mc{G}':=\mc{G} \cup \{\,_K\langle \vek{e}_1, \vek{e}_2, \vek{e}_3 \rangle \}$, then all sums are direct because $\mc{G}'$ is the set of irreducible elements of $\mc{C}$.
\end{ex}

In the following we identify a coordinate ideal $(x_i: i \in I)$ in $K[{\bf x}]$
with the linear subspace $ {}_{K}\langle {\bf
  e}_i,i \in I \rangle$ in $V=K^n$. We then call a set $\mc{C}$ an arrangement
of ideals if the corresponding linear subspaces form an arrangement in $V$.  We
call a set $\mc{G}$ a  building set of ideals if the corresponding subspaces
form a building set in $V$. The sum of two elements $\Id + \Jd$
of an arrangement $\mc{C}$ is direct if and only if $\Id$ and $\Jd$ have
transverse clouds.

\begin{Bem}
One can define building sets and arrangements dually. Then instead of
taking sums of subspaces one takes intersections of subspaces. 
Let $\mc{C}$ be an arrangement in $V$. Then $\mc{P}=\{ V^\bot: V \in
\mc{C} \}$ is an arrangement in the sense of \cite{LiLi, MP} since
$U^\bot \cap W^\bot =(U+W)^\bot$. In particular, for coordinate ideals
$\Id$, we have $\Id^\bot=V(\Id)$. For two ideals $\Id=(x_i: i \in I)$,
$\Jd=(x_j: j \in J)$  we have 
\begin{equation*}
V(\Id + \Jd) = V(x_k: k \in I\cup J)=V(\Id)\cap V(\Jd).
\end{equation*}
One can also generalize the notions of building sets to arrangements of subvarieties of a smooth variety, see \cite{LiLi}.
\end{Bem}

\begin{defn}
Let $\mc{G}$ be a building set of monomial ideals in
$K[x_1,\ldots,x_n]$. For any $\Id\in\mc{G}$ we can write $\Id=( x_a
: a\in A)$ for some subset $A\subset
\left\{1,2,\ldots,n\right\}$. We define the building set of sets
associated to $\mc{G}$ to be the set of all such subsets $A$.
\end{defn}

For any building set of monomial ideals $\mc{G}$ with associated
building set of sets $G$ we generally write $\Id_A$ for the ideal
$( x_a : a\in A)$ for any $A\in G$. By the defining
properties of building sets of ideals we see that $G$ satisfies that
for all sets $A$ and $B$ in $G$ with nonempty intersection also their
union is in $G$. Conversely, given any set of subsets of
$\left\{1,2,\ldots,n\right\}$ satisfying this property, we can obtain
a building set of monomial ideals $\Id_A$, where $A$ runs over the
elements of $G$, for the arrangement that they generate.

\begin{ex}[Mickey Mouse example]
We consider an arrangement of subsets of $\left\{1,2,\ldots,n\right\}$
generated by $A$, $B$ and $C$ where
$A\cap C=\emptyset$, $A\cap B\neq \emptyset$ and $B\cap C\neq
\emptyset$, see figure \ref{mickey}. The smallest building set
is $G=\left\{A,B,C,A\cup B,B\cup C,A\cup B\cup C \right\}$. The
biggest is given by $G'=G\cup \left\{ A\cup C\right\}$.

\begin{figure}[h]
\begin{center}
\begin{pspicture}(-0.2,-1.4)(6,3.5)
\pscircle[fillstyle=vlines](1.3,2.5){0.9}
\pscircle[fillstyle=hlines](3.7,2.5){0.9}
\pscircle(2.5,1){1.5}
\rput(0.1,2.5){{\huge {\bf A}}}
\rput(4.9,2.5){{\huge {\bf C}}}
\rput(2.5,-1){{\huge {\bf B}}}
\end{pspicture}
\caption{Three sets generating an arrangement.}
\label{mickey}
\end{center}
\end{figure}
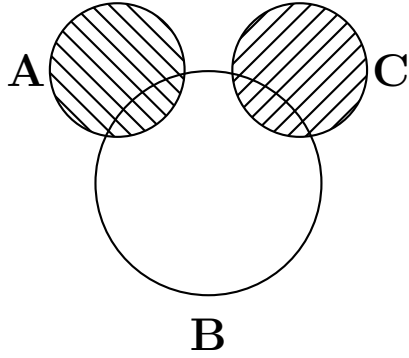

Let us denote for any $D\in G$ the corresponding ideal $\Id_D=(
x_d : d\in D)$. We write $\mc{G}$ for the set of ideals defined
by $G$: $\mc{G}=\left\{\Id_D:D\in G \right\}$. We will show that the
blowup in the building set $\mc{G}$ is smooth; that is, the ideal
$\prod_{D\in G}\Id_D$ is smooth.

The vertices of the
Newton polyhedron of the product $\Id=\prod_{D\in G}\Id_D$ are of the form
${\bf v}=\sum_{D\in G} \nu(D)$, where $\nu(D)={\bf e}_d$ for some
$d\in D$. The function $\nu$ chooses a vertex for each element of the
building set. By Remark \ref{Rmk:summe} we know that $\nu(A\cup B)=\nu(A)$
or $\nu(A\cup B)=\nu(B)$, since else ${\bf v}$ would not be a
vertex. The ideal tangent cone of ${\bf v}$ is given by the sum of the
cones
\begin{equation}
IT_{\nu(D)}(\Id_D) = {}_{\N}\langle {\bf e}_{d} - {\bf
  e}_{\nu(D)}\,,\Sigma \rangle\,,
\end{equation}
where $D$ runs over all elements of $G$ and where $\Sigma$ is the set
of all basis vectors $\left\{{\bf e}_1,\ldots,{\bf e}_n\right\}$\,.

We may assume $\nu(A\cup B)=\nu(A)$ and we first additionally assume
that $\nu(A)$, $\nu(B)$ and $\nu(C)$ are distinct. There are four
possibilities:
\begin{itemize}
\item[(i)] $\nu(B\cup C)=\nu(B)$ and $\nu(A\cup B \cup C)=\nu(A)$.
\item[(ii)] $\nu(B\cup C)=\nu(B)$ and $\nu(A\cup B \cup C)=\nu(B)$.
\item[(iii)] $\nu(B\cup C)=\nu(C)$ and $\nu(A\cup B \cup C)=\nu(A)$.
\item[(iv)]  $\nu(B\cup C)=\nu(C)$ and $\nu(A\cup B \cup C)=\nu(C)$.
\end{itemize}
In case $(i)$ we see that all generators of the form ${\bf e}_i -
\nu(A)$ for $i\notin A$, cannot be minimal since if $i\in B\setminus A$
we have ${\bf e}_i - \nu(A) = {\bf e}_i - \nu(B) + \nu(B)-\nu(A)$ and
similarly for $i\in C$. Hence no $i\in\left\{1,2,\ldots,N\right\}$
exists with two minimal generators of the form ${\bf e}_i - {\bf e}_k$ for some
$1\leq k \leq n$. In case $(ii)$ we see that both $\nu(A) -\nu(B)$ and
$\nu(B)-\nu(A)$ are in the ideal tangent cone, and thus in this case
${\bf v}$ is not a vertex. Case $(iii)$ is similar to the first case;
we can write for any $b\in B$: ${\bf e}_b - \nu(A)= {\bf e}_b - \nu(C)
+ \nu(C) - \nu(A)$, ${\bf e}_b - \nu(C) = {\bf e}_b - \nu(B) + \nu(B)
- \nu(C)$ for any $c\in C$ we can write ${\bf e}_c - \nu(A) = {\bf
  e}_c - \nu(C) +\nu(C) -\nu(A)$. Hence also in this case, the ideal
tangent cone is simplicial. For the last case we have for any $a\in A$:
${\bf e}_a - \nu(C) = {\bf e}_a - \nu(A)+ \nu(A) - \nu(A)$. For any
$b\in B$ we have ${\bf e}_b - \nu(A) = {\bf e}_b - \nu(B) + \nu(B) -
\nu(A)$ and ${\bf e}_b - \nu(C) = {\bf e}_b - \nu(B) + \nu(B) -
\nu(C)$. Hence also in this case we cannot have too many minimal
generators.

The case that $\nu(A)$, $\nu(B)$ and $\nu(C)$ are not all distinct
uses similar arguments and is dealt with easily.
\end{ex}

\begin{Snum} \label{S:buildingsmooth}
Let $\mc{G}$ be a building set of monomial ideals in
$K[x_1,\ldots,x_n]$. Then the ideal $\prod_{\Id\in\mc{G}}\Id$ is
tame.
\end{Snum}

\begin{proof}
Let $G$ be the associated building set of sets associated to
$\mc{G}$. For any $\Id\in \mc{G}$ with $\Id = (x_a:a\in A)$ we also
write $\Id_A$; we thus label the ideals in $\mc{G}$ by their
associated subsets of $\left\{1,2,\ldots,n\right\}$ in $G$. We thus have
$\Id_{A\cup B}=\Id_A+\Id_B$.

Fix a vertex ${\bf v}$ of the Newton polyhedron of
$\prod_{\Id\in\mc{G}}\Id$. Then ${\bf v}$ is sum $\sum_{A\in
  G}\nu(A)$ and $\nu$ has the following property: for $A,B$ in $G$
with $A\cup B$ in $G$, we have $\nu(A\cup B)=\nu(A)$ or $\nu(A\cup
B)=\nu(B)$. The ideal tangent cone of ${\bf v}$ is the sum
\begin{equation}
IT_{{\bf v}}\Bigl(\prod_{\Id\in\mc{G}}\Id\Bigr)=\sum_{A\in G} IT_{\nu(A)}(\Id_A)\,.
\end{equation}

Suppose that $IT_{{\bf v}}\Bigl(\prod_{\Id\in\mc{G}}\Id\Bigr)$ is not
simplicial. Then for some index $a$ there are at least two minimal
generators ${\bf e}_a - {\bf e}_b$ and ${\bf e}_a - {\bf e}_c$, with
$1\leq b < c \leq n$. It follows that we must
have $b=\nu(B)$ and $c=\nu(C)$ for some $B$ and $C$ in $G$ and
furthermore $a$ lies in $B$ and $C$. But then $B\cap C\neq \emptyset$
and thus $B\cup C\in G$. We may assume ${\bf e}_b=\nu(B\cup C)$. Then ${\bf
  e}_c - {\bf e}_b\in IT_{\nu(B)}(\Id_{B\cup C})$ and therefore ${\bf
  e}_c - {\bf e}_b$ lies in $IT_{{\bf v}}
\Bigl(\prod_{\Id\in\mc{G}}\Id\Bigr)$. But then ${\bf e}_a - {\bf e}_b=
{\bf e}_a - {\bf e}_c+ {\bf e}_c - {\bf e}_b$, contradicting that
${\bf e}_a - {\bf e}_b$ is minimal.
\end{proof}


\section{Permutohedral ideals} \label{SECpermutation}


In this section we prove that the so-called permutohedral ideals are tame. The
permutohedral ideal $\Id_{n,k}$  is the ideal in $K[{\bf x}]$ defined by
\begin{equation}\label{perm.prod}
\Id_{n,k}= \prod_{i_1 < \ldots < i_k} (x_{i_1}, \ldots x_{i_k}) , \ \text{all } i_j\in \{1, \ldots , n \}\,.
\end{equation}
 Obviously, the factors in eqn.(\ref{perm.prod}) do not form a building set.
Thus the permutohedral ideals form a class of tame ideals that do not stem from
building sets.

\begin{defn}
Let $p_1, \ldots, p_n$ be real numbers. The permutohedron $P(p_1,\ldots, p_n)$ is the convex polytope in $\R^n$ defined as the convex hull of all permutations of the vector $(p_1, \ldots, p_n)$:
$$ P(p_1, \ldots, p_n)=\conv ((p_{\sigma (1)}, \ldots,  p_{\sigma (n)}): \sigma \in S_n), $$
where $S_n$ is the symmetric group. $P(p_1, \ldots, p_n)$ lies in the
affine hyperplane $H= \{ \vek{x} \in \R^n: \sum_{i=1}^n x_i= \sum_{i=1}^n p_i\}$.
\end{defn}

We consider the permutohedron $\pi_{n,k}=P( \binom{n-1}{k-1}, \binom{n-2}{k-1} ,
\ldots, \binom{k }{ k-1}, 1, \underbrace{0, \ldots,
  0}_{k-1})$. It is easy to see that $\pi_{n,k}$ has
$\tfrac{n!}{(k-1)!}$ vertices whose nonzero entries are pairwise
distinct and that $\pi_{n,k}$ lies in the affine hyperplane $H= \{
\vek{x} \in \R^n: \sum_{i=1}^n x_i= \binom{n}{ k} \}$.

\begin{Lem} \label{L:permecken}
The vertices of the Newton polyhedron of $\Id_{n,k}$ are equal to the vertices of $\pi_{n,k}$.
\end{Lem}

\begin{proof}
We will prove that the vertices of $\pi_{n,k}$ correspond to the vertices of the convex hull of the cloud $I$ of $\Id_{n,k}$. This will prove the lemma as any vertex of $N(\Id_{n,k})$ lies in the cloud $I$. We denote the convex hull of $I$ by $W$.

First let $\vek{p}=( \binom{n-1 }{ k-1}, \binom{n-2}{ k-1} , \ldots,
\binom{k}{k-1}, 1, 0^{k-1})$ be a vertex of $\pi_{n,k}$. We show that
$\vek{p}$ is also a vertex of $W$. Expanding (\ref{perm.prod}) we see
$\vek{p}\in W$. We consider the hyperplane $H_1=\{ \vek{x} \in \R^n:
x_1= \binom{n-1}{ k-1} \}$ and define $W \cap H_1 =W_1$. Then $W_1$ is
nonempty since $\vek{p} \in W_1$. For all $\vek{q} \in W$, $q_1 \leq
\binom{n-1}{k-1}$; this follows from the construction of $W$ as the
sum of the Newton polyhedra $N((x_{i_1}, \ldots x_{i_k}))$. Indeed,
there are exactly $\binom{n-1}{k-1}$ factors that contain
$x_1$. Hence the maximal entry in the $\vek{e}_1$-direction is $\binom{n-1
  }{ k-1}$. Thus $W_1$ is a face of $W$.
Then consider $H_2=\{ \vek{x} \in \R^n: x_2= \binom{n-2}{k-1} \}$ and
set $W_2=W_1 \cap H_2$. Similar to above we see $H_2$ lies on one side
of $W_2$; in the factors of $\Id_{n,k}$ we have chosen $\binom{n-1}{
  k-1}$ times $x_1$, and of the remaining factors $\binom{n-2}{ k-1}$
contain $x_2$. Hence for any $\vek{q}\in W$, $q_2$ can at most be
$\binom{n-2 }{k-1}$. Continuing this way we find that $W_{n-k} \cap
H_{n-k+1}=\vek{p}$ with $H_{n-k+1}=\{\vek{x} \in \R^n: x_{n-k+1}=1
\}$. Hence $\vek{p}$ is also a vertex of $W$.

We call a vertex $\vek{p}$ of $\pi_{n,k}$ a max-vector. Fix any point
$\vek{v}$ in $W$ that is not a max-vector. We will show that $\vek{v}$
can be written as a convex combination of other points contained in
$W$, showing that $\vek{v}$ is not a vertex.

We start with the case $k=2$. For $k=2$ a max-vector is a
permutation of $(n-1,\ldots, 1, 0)$. The point
$\vek{v}$ corresponds to a monomial appearing in the expansion of
$(x_1,x_2)(x_1,x_3) \cdots (x_{n-1},x_n)$. Such a monomial arises by
choosing from each pair $(x_i,x_j)$ the $x_i$ or the $x_j$. For
generic $\vek{v}$ there are many different possible choices that all
contribute to $\vek{v}$. We translate this choosing of terms from each
factor into different graphs: Consider $n$ vertices labeled with the
numbers $1,\ldots,n$. To codify the choice we direct an edge from $i$
to $j$ if in the factor $(x_i,x_j)$ the monomial $x_i$ was
chosen. Thus for each way of choosing one out of each $(x_i,x_j)$ we
get a graph, which is a complete directed graph with $n$
vertices. For any given graph associated to $\vek{v}$, $v_i$ is the
number of outgoing edges from vertex $i$. For a max-vector there is
only one graph, which is a tree, see for example the (unique) graph
corresponding to $(4,3,2,1,0)$ for $n=5$ below:
\begin{center}
\begin{tabular}{c}
\begin{pspicture}(-2,-1)(3,4)
\rput(0,0){\rnode{4}{\pscirclebox{4}}}
\rput(2,0){\rnode{3}{\pscirclebox{3}}}
\rput(-0.61,1.9){\rnode{5}{\pscirclebox{5}}}
\rput(2.61,1.9){\rnode{2}{\pscirclebox{2}}}
\rput(1,3.2){\rnode{1}{\pscirclebox{1}}}
\ncline{->}{1}{2}\ncline{->}{1}{3}\ncline{->}{1}{4}
\ncline{->}{1}{5}\ncline{->}{2}{3}\ncline{->}{2}{4}
\ncline{->}{2}{5}\ncline{->}{3}{4}\ncline{->}{3}{5}
\ncline{->}{4}{5}
\end{pspicture}
\end{tabular}
\end{center}

If we flip the orientation of an edge $i \to j$ in a graph associated
to $\vek{v}$ we choose $x_j$ instead of $x_i$ in the factor
$(x_i,x_j)$. The resulting graph corresponds to the graph of the point
$\vek{v}'=(v_1, \ldots, v_i-1, \ldots, v_j +1, \ldots, v_n)$. Clearly
$\vek{v}' \in W$.

Claim: If $\vek{v} \in W$ is not a max-vector, then any of its associated graphs contains a cycle.

Proof of Claim: Let $G$ be any of the graphs associated to
$\vek{v}$. We may assume $v_1\geq v_2\geq \ldots \geq v_n$. If
$\vek{v}$ is not a max-vector, there is a $j$ with $v_j< n-j$. It
follows that no vertex from $1$ to $j-1$ can be contained in a cycle,
since $1$ only has outgoing edges, $2$ has one incoming edge from $1$
and for the rest outgoing, and so on, till we get to $j$. Thus we can
restrict the search for cycles to the subgraph of $G$ with vertices
$j, \ldots, n$ and may assume that $v_1<n-1$.

We can find a cycle with the following procedure: since $v_1<n-1$, the
vertex $1$ has at least one incoming edge from a vertex $l$. We go
from $1$ to $l$. Since $v_l\leq v_1<n-1$ also this vertex will have at
least one incoming edge. We choose such an incoming edge and go to the
next vertex. Since the graph has a finite number of vertices, we will
meet a certain vertex for a second time after a finite number of
steps. Hence we will find a cycle by retracing a part of our
path. This proves the claim. 

For $n=5$ we depicted a graph associated to the vector $(4,2,2,1,1)$, containing several cycles:
\begin{center}
\begin{tabular}{c}
\begin{pspicture}(-2,-1)(3,4)
\rput(0,0){\rnode{4}{\pscirclebox{4}}}
\rput(2,0){\rnode{3}{\pscirclebox{3}}}
\rput(-0.61,1.9){\rnode{5}{\pscirclebox{5}}}
\rput(2.61,1.9){\rnode{2}{\pscirclebox{2}}}
\rput(1,3.2){\rnode{1}{\pscirclebox{1}}}
\ncline{->}{1}{2}\ncline{->}{1}{3}\ncline{->}{1}{4}
\ncline{->}{1}{5}\ncline{->}{2}{3}\ncline{->}{2}{4}
\ncline{->}{3}{4}\ncline{->}{3}{5}\ncline{->}{4}{5}
\ncline{->}{5}{2}
\end{pspicture}
\end{tabular}
\end{center}

We may assume that the cycle is $C=1 \to 2 \to \ldots \to n$. If we flip the orientation of the edge $1 \to 2$ we see that $\vek{v}^{(1)}=(v_1 -1, v_2 +1, v_3, \ldots, v_n)$ is in $W$. Flipping the edges $j \to j+1$ we get that the vectors $\vek{v}^{(2)}=(v_1,v_2-1, v_3,+1, v_4, \ldots, v_n), \ldots, \vek{v}^{(n)}=(v_1+1,v_2, v_3, v_4, \ldots,v_n, v_n-1)$ all lie in $W$. It is easy to see that
\begin{equation}
\vek{v}=\frac{1}{n} \sum_{i=1}^n \vek{v}^{(n)}\,.
\end{equation}
Hence $\vek{v}$ is contained in $\conv ( \vek{v}^{(1)},\vek{v}^{(2)}, \ldots, \vek{v}^{(n)})$ and $\vek{v}$ is not a vertex of $N$.

The general case $k \geq 3$ is similar: to each factor $(x_{i_1},
\ldots, x_{i_k})$ corresponds a $k$-simplex with vertices $i_1, \ldots
i_k$. For example, for $n=4$ and $k=2$ one considers the faces of the
tetrahedron.

\begin{center}
\begin{tabular}{c}
\begin{pspicture}(1.4,0)(6,4.1)
\psset{unit=0.7cm}
\pscircle(2,2){0.5}
\pscircle(4,1){0.5}
\pscircle(5.7,2.3){0.5}
\pscircle(3.7,4){0.5}
\rput(2,2){$\mathbf{1}$}\rput(4,1){$\mathbf{2}$}\rput(5.7,2.3){$\mathbf{3}$}\rput(3.7,4){$\mathbf{4}$}
\qline(2.4,1.7)(3.55,1.2)
\qline(4.38,1.32)(5.3,2.0)
\psline[linestyle=dashed](2.48,2.04)(5.2,2.25)
\qline(2.34,2.37)(3.338,3.658)
\qline(4.0,1.5)(3.78,3.52)
\qline(5.3,2.6)(4.1,3.7)
\end{pspicture}
\end{tabular}
\end{center}

Like above, to $\vek{v}$ correspond choices of some $x_{i_j}$
from each factor $(x_{i_1}, \ldots, x_{i_k})$.
If we choose $x_{i_j}$ in a simplex $(x_{i_1}, \ldots, x_{i_k})$, we
mark the vertex $i_j$ as the only ``outgoing'' vertex and the other
$k-1$ vertices as the ``incoming'' vertices of the $k$-simplex. We
search for cycles of vertices on the $k$-simplices in the following
way: start with the $k$-simplex $(1, \ldots, k)$ where we mark the
outgoing vertex. Suppose $1$ is the outgoing vertex. Then we can leave
through $1$ to another $k$-simplex if $1$ is incoming there. We can
find a chain of vertices (and $k$-simplices) representing any point
$\vek{v}$ in $N$. First let $\vek{v}$ be the lexicographically ordered
max-vector $(\binom{n-1}{k-1},\binom{n-2}{k-1}, \ldots, 1, 0,
\ldots, 0)$. Then $1$ is an outgoing vertex in all $\binom{n-1}{ k-1}$
simplices containing it. Since $1$ is nowhere incoming we cannot find
a cycle containing $1$. The vertex $2$ is outgoing for $\binom{n-2}{k-1}$
simplices and incoming for $\binom{n-2 }{k-2}$. But $2$ is
incoming only from simplices with outgoing $1$. Thus $2$ can also not
be part of a cycle. Continuing this argument we see that no vertex $1,
\ldots, n-k+1$ can be part of a cycle. But the remaining vertices are
only incoming and are also not in a cycle.

If $\vek{v} \in W$ is not a max-vector we can assume that $v_1=\binom{n-1}{
k-1}, \ldots, v_{i-1}=\binom{n-i+1}{k-1}, v_i < \binom{n-i}{ k-1}$. Then all
vertices $1, \ldots, i-1$ are not part of a
cycle. But $i$ is at least incoming for one simplex. By a similar
method as for $k=2$ we can find a cycle $i_1 \to i_2 \to \ldots \to i_m$
of length $m$ of vertices on the $k$-simplices for $i_j \geq i$ and
find $m$ points $\vek{v}^{(i)}$ in $W$ such that $\vek{v}$ lies in
their convex hull. As an illustration of the method, we depicted the
$2$-simplices of the tetrahedron and marked the outgoing vertices with
an asterisk:

\begin{tabular}{cccc}
\begin{pspicture}(0.3,0)(3.5,3.0)
\psset{unit=0.5cm}
\pscircle(1,1){0.5}\rput(1,1){$\mathbf{2}$}
\pscircle(4,1){0.5}\rput(4,1){$\mathbf{3}$}
\pscircle(2.5,3.5){0.5}\rput(2.5,3.5){$\mathbf{1}$}
\qline(1.5,1)(3.5,1)\qline(1.224,1.424)(2.2,3.13)\qline(3.75,1.43)(2.8,3.1)
\rput(2.5,2.7){$\mathbf{\star}$}
\end{pspicture}
&
\begin{pspicture}(0.3,0)(3.5,3.0)
\psset{unit=0.5cm}
\pscircle(1,1){0.5}\rput(1,1){$\mathbf{3}$}
\pscircle(4,1){0.5}\rput(4,1){$\mathbf{1}$}
\pscircle(2.5,3.5){0.5}\rput(2.5,3.5){$\mathbf{4}$}
\qline(1.5,1)(3.5,1)\qline(1.224,1.424)(2.2,3.13)\qline(3.75,1.43)(2.8,3.1)
\rput(3.38,1.38){$\mathbf{\star}$}
\end{pspicture}
&
\begin{pspicture}(0.3,0)(3.5,3.0)
\psset{unit=0.5cm}
\pscircle(1,1){0.5}\rput(1,1){$\mathbf{3}$}
\pscircle(4,1){0.5}\rput(4,1){$\mathbf{4}$}
\pscircle(2.5,3.5){0.5}\rput(2.5,3.5){$\mathbf{2}$}
\qline(1.5,1)(3.5,1)\qline(1.224,1.424)(2.2,3.13)\qline(3.75,1.43)(2.8,3.1)
\rput(2.5,2.7){$\mathbf{\star}$}
\end{pspicture}
&
\begin{pspicture}(0.3,0)(3.5,3.0)
\psset{unit=0.5cm}
\pscircle(1,1){0.5}\rput(1,1){$\mathbf{4}$}
\pscircle(4,1){0.5}\rput(4,1){$\mathbf{2}$}
\pscircle(2.5,3.5){0.5}\rput(2.5,3.5){$\mathbf{1}$}
\qline(1.5,1)(3.5,1)\qline(1.224,1.424)(2.2,3.13)\qline(3.75,1.43)(2.8,3.1)
\rput(3.38,1.38){$\mathbf{\star}$}
\end{pspicture}
\\
\end{tabular}

The graph corresponds to the vector $(2,2,0,0)$, which is not a
max-vector. We can find a cycle $1\to 2 \to 1$ using the $2$-simplices
$(1,2,3)$ and $(1,2,4)$.
\end{proof}

\begin{Bem}
The permutation polynomial $p_{n,k}$ is defined by
\begin{equation}
p_{n,k}=\prod_{1\leq i_1<\ldots < i_k\leq n}(x_{i_1}+ \ldots + x_{i_k})\,.
\end{equation}
When we expand $p_{n,k}=\sum_{{\bf a}}c_{{\bf a}}{\bf x}^{{\bf a}}$ in
monomials, then the vectors ${\bf a}\in \N^n$ with $c_{{\bf a}}$
nonzero define a finite set $M$. The vertices of the convex hull of
$M$ correspond to the max-vectors introduced in the proof of Lemma
\ref{L:permecken}. The max-vectors are precisely those ${\bf a}\in M$
with $c_{{\bf a}}=1$.
\end{Bem}

\begin{Snum}
Each ideal $\Id_{n,k}$ is tame.
\end{Snum}

\begin{proof}
Let $I$ the cloud of $\Id_{n,k}$. By Lemma \ref{L:permecken}, the vertices of
$I$ are all permutations of $\vek{b}=(\binom{n-1}{k-1}, \ldots, 1, 0,
\ldots , 0)$. We first show that the elements ${\bf e}_1$, ${\bf e}_{i+1} - {\bf
e}_i$ for $1\leq i \leq n-k$ and ${\bf e}_{n-k+j} - {\bf e}_{n-k+1}$ for $2\leq
j \leq k$ are all in $\IT{\vek{b}}{\Id_{n,k}}$.

From Lemma \ref{Lunimodular} we know that the minimal generators of
$\IT{\vek{b}}{\Id_{n,k}}$ are of the form $\vek{e}_i$ or $\vek{e}_i -
\vek{e}_j$. There cannot be a vector $\vek{e}_1 - \vek{e}_j$, $j \geq 2$ in
$IT_{\vek{b}}(\Id_{n,k})$ because this generator would stem from a point
$\vek{b} + (\vek{e}_1 - \vek{e}_j)= (\binom{n-1 }{ k-1}+1, \ldots, 0)$ in $I$.
But this contradicts $x_1 \leq b_1$ for all $\vek{x}$ in $I$. But then
$\vek{e}_1$ is a minimal generator of $\IT{\vek{b}}{\Id_{n,k}}$.

For the vectors $\vek{e}_i - \vek{e}_{i-1}, i=2, \ldots, n-k+2$ consider 
the vector $\vek{a}=\vek{b} + \vek{e}_i - \vek{e}_{i-1}$; ${\bf a}$ has the same
coordinates as $\vek{b}$ except for $a_{i-1}=b_{i-1} -1$ and $a_i=b_i+1$. The
vector ${\bf b}$ is a max-vector and $b_1>b_2>b_3>\ldots$. Hence in some factor
of the product defining $\Id_{n,k}$ there exist $x_i$ and $x_{i-1}$ and we have
chosen $x_{i-1}$ in this factor. Choosing $x_{i}$ instead of $x_{i-1}$ we see
${\bf a}\in I$. Hence  $\vek{e}_i - \vek{e}_{i-1}$, $i=2, \ldots, n-k+2$ are in
$\IT{\vek{b}}{\Id_{n,k}}$. Practically the same argument applies to conclude
that the vectors  $\vek{e}_{n-k+j} - \vek{e}_{n-k+1}$ are in
$\IT{\vek{b}}{\Id_{n,k}}$.

From the above we see that all vectors ${\bf e}_i - {\bf e}_j$ with $i>j$ and
$j\leq n-k+1$ are in $IT_{{\bf b}}(\Id_{n,k})$. If ${\bf v}\in IT_{{\bf
b}}(\Id_{n,k})$, then $v_l \geq 0$ for all $l>n-k+1$, since $b_l = 0$ for all
$l>n-k+1$. It follows that if ${\bf e}_i - {\bf e}_j$ and ${\bf e}_i - {\bf
e}_{j'}$, with $j>j'$ are two minimal generators, then $j,j'\leq n-k+1$ and
since $IT_{{\bf b}}(\Id_{n,k})$ is pointed, it follows that $i>j,j'$. But then
${\bf e}_i - {\bf e}_j + {\bf e}_j - {\bf e}_{j'}= {\bf e}_i - {\bf e}_{j'}$
shows that ${\bf e}_i - {\bf e}_{j'}$ is not minimal. Hence $IT_{{\bf
b}}(\Id_{n,k})$ is simplicial.
\end{proof}

\begin{Bem}
The proof for $\Id_{n,2}$ can be done by applying Lemma
  \ref{LpaarweiseVereinigung} to the ideals $\Id_i =(x_i)$; also see
  Example \ref{Bsp.planes.perm}.
\end{Bem}

\bibliographystyle{martijn2.bst}
\bibliography{biblioEF}

\vskip 2cm
Eleonore Faber, Dennis Bouke Westra\\
Fakult\"at f\"ur Mathematik\\
Universit\"at Wien\\
Nordbergstra{\ss}e 15\\
1090 Wien\\
{\tt eleonore.faber@univie.ac.at}, {\tt dennis.bouke.westra@univie.ac.at}

\end{document}